\documentclass[onefignum,onetabnum]{siamart171218}

\usepackage{xcolor}
\usepackage{lipsum}
\usepackage{amsfonts}
\usepackage{mathtools}
\usepackage{bbm}
\usepackage{graphicx}
\usepackage{epstopdf}
\usepackage{algorithmic}
\usepackage{stackrel}
\usepackage{caption}
\usepackage{subcaption}
\usepackage[rightcaption]{sidecap}
\usepackage{comment}
\usepackage[normalem]{ulem} 
\ifpdf
  \DeclareGraphicsExtensions{.eps,.pdf,.png,.jpg}
\else
  \DeclareGraphicsExtensions{.eps}
\fi
\crefname{subsection}{section}{sections}
\crefname{ALC@unique}{step}{steps}

\newsiamremark{remark}{Remark}
\newsiamremark{assumption}{Assumption}
\crefname{assumption}{Assumption}{Assumptions}
\newsiamremark{note}{Note}
\newsiamremark{hypothesis}{Hypothesis}
\crefname{hypothesis}{Hypothesis}{Hypotheses}
\newsiamthm{claim}{Claim}

\headers{Solving CCPs via a Smooth Sample-Based Nonlinear Approximation}{A. Pe\~na-Ordieres, J. Luedtke, and A. W\"achter}

\title{Solving Chance-Constrained Problems via a Smooth Sample-Based Nonlinear Approximation\thanks{Submitted to the editors DATE.}}

\author{Alejandra Pe\~na-Ordieres\thanks{Department of Industrial Engineering and Management Sciences, Northwestern University, Evanston, IL 
  (\email{alejandra.pena@u.northwestern.edu}, \email{andreas.waechter@northwestern.edu}).
  These authors were funded in part by National Science Foundation grant DMS-1522747.}
\and
James R. Luedtke\thanks{Department of Industrial and Systems Engineering, University of Wisconsin-Madison, Madison, WI 
  (\email{jim.luedtke@wisc.edu}).
  This author was funded by the U.S. Department of Energy, Office of Science, Office of Advanced Scientific Computing Research (ASCR) under Contract DE-AC02-06CH11347.}
\and
Andreas W\"achter\footnotemark[2]}

\usepackage{amsopn}

\ifpdf
\hypersetup{
  pdftitle={Solving CCPs via a Smooth Sample-Based Nonlinear Approximation},
  pdfauthor={A. Pe\~na-Ordieres, J. Luedtke, and A. W\"achter}
}
\fi

\DeclareMathOperator{\argmax}{argmax}

\DeclareMathOperator{\E}{E}
\DeclareMathOperator{\proba}{\mathbb{P}}

\graphicspath{ {images/} }

\begin{document}

\maketitle

\begin{abstract}
We introduce a new method for solving nonlinear continuous optimization problems with chance constraints. Our method is based on a reformulation of the probabilistic constraint as a quantile function. The quantile function is approximated via a differentiable sample average approximation. We provide theoretical statistical guarantees of the approximation, and illustrate empirically that the reformulation can be directly used by standard nonlinear optimization solvers in the case of single chance constraints. Furthermore, we propose an S$\ell_1$QP-type trust-region method to solve instances with joint chance constraints. We demonstrate the performance of the method on several problems, and show that it scales well with the sample size and that the smoothing can be used to counteract the bias in the chance constraint approximation induced by the sample approximation.
\end{abstract}




\section{Introduction} \label{sec: Intro}

Consider the following minimization problem subject to a chance constraint
\begin{subequations} \label{eq: CCP}
\begin{align}
\min_{x \in X} &\quad f(x)\\
\mathrm{s.t.} &\quad \proba(c_j(x,\xi) \leq 0, j \in \{1,\ldots,m\}) \geq 1 - \alpha, \label{eq: ProbaConst}
\end{align}
\end{subequations}
where $X \subseteq \mathbb{R}^n$ represents a closed deterministic feasible region, $f: \mathbb{R}^n \rightarrow \mathbb{R}$ represents the objective function to be minimized, $c: \mathbb{R}^n \times \mathbb{R}^s \rightarrow \mathbb{R}^m$ is a vector-valued function, $\xi$ is a random vector with support $\Xi \subseteq \mathbb{R}^s$, and $\alpha \in (0,1)$ is a risk parameter given a priori. This kind of problem is known in the literature as a joint chance-constrained problem (JCCP), in which all $m$ constraints must be satisfied simultaneously with probability at least $1-\alpha$. A special case is when $m=1$ and only one constraint needs to be satisfied with probability $1-\alpha$; this is called a single chance-constrained problem (SCCP).

Chance-constrained optimization problems were introduced in \cite{CharCoopSymo58} and have been extensively studied ever since; see e.g., \cite{Prek03}. These problems can be found in a plethora of applications such as water management \cite{LodiMalaNann16}, optimization of chemical processes \cite{HenrLiMoll01,HenrMoll03}, and optimal power flow \cite{BienCherHarn14,MolzRoal18}, to mention a few.

This formulation gives rise to several difficulties. First, the structural properties of $c(x,\xi)$ may not be passed to the constraint $\proba(c_j(x,\xi) \leq 0, j \in \{1,\ldots,m\}) \geq 1 - \alpha$. For example, even if $c_1(x,\xi),\dots,c_m(x,\xi)$ are all linear in $x$, the probabilistic constraint may not define a convex feasible region. Second, the functional form of the distribution of $\xi$ may not be known. And finally, even if the distribution of $\xi$ is known, there might not be a tractable analytical function to describe the constraint or a numerically efficient way to compute it.

We focus on problems in which $f$ and $c$ are continuous and differentiable, with respect to $x$, and where $\xi$ is such that the random variable defined by $c(x,\xi)$ has a continuous cumulative density function (cdf) for all $x \in X$. Other than that, we make no assumptions on the type of constraint $c(x,\xi)$; we do not limit ourselves to linear or convex functions.

\subsection{Literature Review}

In this section we review several methods for solving chance-constrained problems. The first couple of methods make assumptions on the distribution of $\xi$, or at least require that some knowledge of the distribution $\xi$ is available, whereas the latter methods make no such assumptions. Our idea falls into the latter category, and does not impose restrictions on the form of the constraints or on the distribution, except that it is a continuous distribution.

To compute \cref{eq: ProbaConst}, some methods take advantage of the special structure of constraint $c$ and the distribution of $\xi$. For example, in \cite[Lemma 2.2]{Henr07}, if $\xi$ has an elliptical symmetric distribution (e.g., multivariate normal random variable) and if $c(x,\xi) = \xi^Tx-b$, then the problem can be expressed as a second-order cone program, and can be solved in a tractable and efficient manner using standard solvers. More sophisticated procedures make use of efficient numerical integration techniques to compute Gausian-like probabilities, as well as their corresponding gradients (e.g., \cite{Deak00,GenzBret09,hantoute2019subdifferential,VanackAlekMuno18}). For linear chance constraints, approaches based on cutting planes \cite{VanacketAl14}, bundle \cite{VanackSaga14}, and SQP \cite{BremHenrMoel15} methods have been developed. Nonlinear constraints can be addressed using a spherical-radial decomposition method \cite{Heit19,VanackHenr14}.


There are several approaches that avoid making assumptions on $\xi$ by sampling to approximate the probabilistic constraint. We divide such techniques into those that use an integer reformulation, and those that are structured as nonlinear programming problems (NLP).

In the first category, the probabilistic constraint is replaced by the empirical distribution function and formulated as a mixed-integer nonlinear program (MINLP), e.g.\ \cite{LuedAhme08}. This method takes realizations of the random variable $\xi$ and requires that at least $(1-\alpha)\%$ of the realized constraints are satisfied. This idea has several advantages; the solution converges to the true solution of \cref{eq: CCP} as the sample size increases, and for convex instances, a global solution of the approximation can be computed by branch and bound. However, the complexity of the problem increases with the sample size because the number of binary variables increases. This issue has been circumvented to some extent for certain types of constraints $c(x,\xi)$, e.g. linear constraints \cite{LuedAhmeNemh10}, but this procedure cannot be applied to general constraints.

To avoid the need for solving MINLPs, a conservative and more tractable approximation can be obtained by using the scenario-based approach, e.g. \cite{CalaCamp05,CampGaraPran09,NemiShap06_SA}. These methods seek to satisfy all of the constraints in a sample of pre-specified size to ensure that the solution is feasible with high confidence. The reformulation proposed in~\cite{CalaCamp05} allows one to obtain solutions very easily since the structure of the original constraints $c(x,\xi)$ is preserved. For example, if $f$ is convex, $X$ defines a convex set and $c(x,\xi_i)$ define convex functions on $x$ for all $i=\{1,\dots,N \}$ (where $N$ is the size of the sample), then the problem that arises is convex. However, the limitation of this method is that one may obtain highly conservative solutions and that it does not offer direct control on the probability level of the chance constraint.

Another tractable conservative approximation can be obtained by using conditional value-at-risk, which is based on using a conservative convex approximation of the indicator function \cite{NemiShap06}. This approximation typically finds feasible but sub-optimal solutions. To avoid overly conservative solutions, some authors propose a difference-of-convex functions approximation to the indicator function \cite{HongYangZhan11}, while some others propose outer or inner approximations to the indicator function \cite{CaoZava17, GeleHoffKlop17}. The main drawback of the last two methods mentioned is that the constraint function can become very ``flat'' in the regions of interest. That is, when the probability of the constraint satisfaction evaluated at a given $x$ is close to one or close to zero, the information obtained from the gradient may not provide a sufficiently accurate approximation of the constraints or may even be erroneous. We discuss this issue further when we motivate our proposed reformulation of \cref{eq: CCP} in \cref{sec: VaRReformulation}.

A potential drawback of sample-based methods is that in some cases the number of samples needed to obtain an accurate solution might be fairly large \cite[Example 1]{Henr11}. Tailored approaches that take the problem structure directly into account can be more efficient when applicable.




\subsection{Our Approach}

In this paper we propose an NLP reformulation of the problem based on sample average approximation (SAA), in which we rewrite the probability constraint as a quantile constraint. We first argue that the latter formulation is more suitable for gradient-based optimization methods than the former. Then, we present a way to approximate the quantile constraint using samples in a way that provides a smooth constraint. We find in our numerical experiments that the smoothing helps control the variance of the solutions obtained from different samples. The approximation we present uses a parameter $\epsilon$ that depends on the scaling of the problem, which might be difficult to select; thus, we also discuss an algorithm that finds an appropriate value of $\epsilon$.

The proposed approach has the following advantages. First, we observe empirically that the solutions we obtain are very often ``more robust'' compared to solutions obtained from solving the traditional (non-smoothed) SAA problem via mixed-integer programming (MIP) formulations that use the empirical cdf. More specifically, when we compare the solutions obtained from our method to those returned by the MIP, we consistently find that our method attains a better objective function value and has less variability across different samples, while still maintaining out-of-sample feasibility. We attribute this to the fact that our approach uses information from the entire sample, whereas the empirical cdf only takes into account whether enough constraints are satisfied. These numerical experiments suggest that the solutions we attain are able to ``generalize'', i.e., the solutions obtain better performance on out-of-sample scenarios. Second, our method is scalable; we find that the time to solve an instance with our approach grows modestly with the sample size. Third, our formulation does not assume any type of structure from the constraint functions save for continuity and differentiability with respect to $x$; therefore, we are able to find local solutions to non-linear and non-convex problems. Finally, for SCCPs, a standard NLP solver can be used to find the solution of the problem for the appropriate choice of the parameters involved in the formulation, and for JCCPs we are able to derive an S$\ell_1$QP algorithm that has proven empirically to give fast convergence to the solution after only a few iterations.

The remainder of this paper is organized as follows. In \cref{sec: VaRReformulation} we present the proposed reformulation of the problem. In \cref{sec: Feasibility} we study the convergence and feasibility of the solutions to the approximate problem to those of the true problem \cref{eq: CCP}. In \cref{sec: SCCP} we discuss single chance-constrained problems, and in \cref{sec: JCCP} we extend this approach to solve joint chance-constrained problems. In \cref{sec: NumRes} we present numerical results. Finally, we make some concluding remarks in \cref{sec:concl}.

\section{Reformulation and approximation of the CCP} \label{sec: VaRReformulation}

The quantile of a random variable $Y$ at the level $1 - \alpha \in (0,1)$ is defined as
$$Q^{1 - \alpha}(Y) = \inf \{ y \in \mathbb{R} \mid \proba(Y \leq y) \geq 1 - \alpha \}.$$
Therefore, $\proba(C(x,\xi) \leq 0) \geq 1 - \alpha$ is equivalent to $Q^{1 - \alpha}(C(x,\xi)) \leq 0$, where $C: \mathbb{R}^n \times \mathbb{R}^s \rightarrow \mathbb{R}$ is a real valued function, $x \in \mathbb{R}^n$, and $\xi$ is a random vector taking values in $\mathbb{R}^s$. Here, we make a distinction between $c(x,\xi)$ and $C(x,\xi)$; $C(x,\xi)$ denotes a scalar-valued function and $c(x,\xi)$ is the vector-valued function defined in problem \cref{eq: CCP}. The reason for this distinction is that $Q^{1 - \alpha}(C(x,\xi))$ is only well defined if $C(x,\xi)$ is a scalar random variable. However, when $c$ is a vector-valued constraint, we set $C(x,\xi) = \max_{j = 1,\ldots,m} c_j(x,\xi)$ and obtain the following reformulation of problem \cref{eq: CCP},
\begin{align}
\min_{x \in X} &\quad f(x), \nonumber \\
\mathrm{s.t.} &\quad Q^{1 - \alpha}(C(x,\xi)) \leq 0. \nonumber
\end{align}

The advantage of rewriting the constraint as the quantile is that this reduces the ``flatness'' encountered when the probability function $\proba(C(x,\xi) \leq 0)$ is used as the constraint, see \cref{fig: ProbAndVaR} as an example. Now the feasible region is measured in the image of $C(x,\xi)$ and not in the bounded space $[0,1]$, which is the image of the probability function.

\begin{SCfigure}[][htbp]
\centering
  \includegraphics[width=0.5\linewidth]{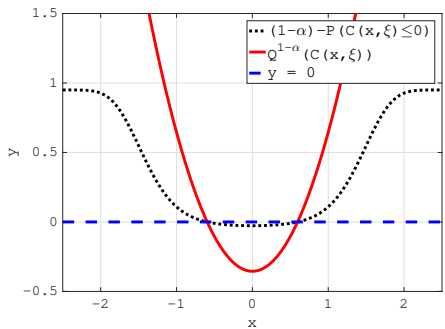}
  \caption{Comparison of $Q^{1-\alpha}(C(x,\xi))$ and $(1 - \alpha) - P(C(x,\xi) \leq 0)$ where $C(x,\xi) = x^2 - 2 + \xi$, $\xi \sim N(0,1)$, and $\alpha = 0.05$.}
  \label{fig: ProbAndVaR}
\end{SCfigure}

\cref{fig: ProbAndVaR} also shows how the feasible regions coincide for both formulations. Yet, although the two formulations are equivalent, solvers that depend on gradient information are likely to perform better when the quantile function is used, as we will see in the numerical experiments. This can be attributed to the fact that the linearization of the quantile function provides an approximation of the constraint that is reasonably accurate for a wider range than the linearization of the probabilistic constraint. The consequence of ``flatness'' is that a non-specialized gradient-based solver that starts at a point $x$ where $\proba(C(x,\xi)\leq 0) \approx 1$ takes a big step to decrease the objective, disregarding the constraint function. Thus, the solver is forced to reject many trial steps, which can lead to an increased number of iterations. {This observation has been made before in \cite[p.734]{VanackSaga14}, where the authors state that the chance constraint ``can define a gully-shaped feasible region that makes [the chance-constrained problem] hard to tackle with standard nonlinear programming solvers \cite{Mayer00}.'' Moreover, \cite{Mayer00} shows that this characteristic becomes more noticeable, at least for normal random variables, as the dimension of the random variable increases.}

To illustrate this claim, we compared the numerical performance of the nonlinear programming solver \texttt{Knitro} for each of the two formulations for the constraints $C(x,\xi) = x^2 - 2 + \xi$. We used $\texttt{Knitro}$ to maximize $x$ using: (1) the probabilistic constraint $\proba(C(x,\xi)\leq 0) = \Phi(2 - x^2) \geq 1 - \alpha$, and (2) the quantile constraint $Q^{1 - \alpha}(C(x,\xi)) = x^2 - 2 + \Phi^{-1}(1 - \alpha) \leq 0$; see \cref{fig: ProbAndVaR}. We chose $x_0 = 3$ as starting point. \cref{tab: ProbVsVaRAnlytic} illustrates the performance in terms of iteration count using the two types of constraints. Each column shows the number of iterations using different bounds on $x$. The results in the table suggest that the quantile function is more robust since the number of iterations is independent of the bounds and the solver always returned feasible and optimal solutions.

\begin{table}[htbp]
\centering
\begin{tabular}{c|cccc}
Bounds    & $[-1,1]$ & $[-10,10]$ & $[-100,100]$ & $(-\infty,\infty)$ \\ \hline\hline
$\proba(C(x,\xi)\leq 0) \geq 1-\alpha$ & 6        & 36         & 14           & $5^{*}$\\
$Q(C(x,\xi)) \leq 0$    & 6        & 6          & 6            & 6
\end{tabular}
\caption{Iterations performed by \texttt{Knitro} when maximizing $x$ using two different formulations to describe the feasible region and different bounds on $x$. $^*$: \texttt{Knitro} gives the following error message: \emph{``Convergence to an infeasible point. Problem appears to be locally infeasible.''}} \label{tab: ProbVsVaRAnlytic}
\end{table}


Motivated by the aforementioned arguments, we propose to work with the quantile formulation. However, as with the probabilistic formulation, an explicit expression to compute the quantile is usually not available or can be intractable to compute. To address this issue, we introduce a sample average approximation of the quantile in the following section.

Alternative approaches to overcome the convergence failures arising from the flatness of the probability constraints include a logarithmic transformation \cite{Mayer00} or the use of interpolation steps \cite{Szan88} within a tailored optimization algorithm.

\subsection{Smooth approximation of the quantile function} \label{sec: SmoothVaR}

A quantile estimator can be obtained by inverting an estimator of the cdf when the distribution is continuous. Let $\{ \xi_1,\ldots,\xi_N \}$ be an independent and identically distributed (i.i.d.) sample and consider the empirical cdf of $\{ C(x,\xi_1),C(x,\xi_2),\ldots, C(x,\xi_N) \}$ at a given point $x$,
$$\widetilde{F}^N(t;x) = \frac{1}{N}\sum_{i = 1}^N \mathbbm{1}(C(x,\xi_i) \leq t).$$
An estimate of the $(1 - \alpha)$-quantile (and hence $Q^{1-\alpha}(C(x,\xi))$) is obtained from a value $t$ such that $\widetilde{F}^N(t;x) \approx 1 - \alpha$. The $(1-\alpha)$-empirical quantile at $x$ is defined as
$$\widetilde{Q}^{1-\alpha}(C^N(x)) = \inf \left\lbrace y \mid \frac{1}{N} \sum_{i = 1}^N \mathbbm{1}\left( C(x,\xi_i) \leq y \right) \geq 1 - \alpha \right\rbrace = C_{[M]}(x),$$
for $M = \lceil (1 - \alpha) N \rceil$. Here, $\mathbbm{1}$ is the indicator function and $C_{[j]}$ denotes the $j$th smallest observation of the values $\{C(x,\xi_1),\ldots,C(x,\xi_N)\}$ for a fixed $x$.

There are two main drawbacks to using the empirical quantile as an approximation. First, the empirical approximation is not differentiable. The ``active scenario'' (the scenario $\xi_i$ such that $C_{[M]}(x) = C(x,\xi_i)$) changes for different values of $x$. Hence, even if the constraint $C(x,\xi)$ is smooth for a fixed value of $\xi$, this does not hold for $\widetilde{Q}^{1-\alpha}(C^N(x))$. Second, the aforementioned feature also introduces artificial local minima to the problem, as can be seen in \cref{fig: EmpTrue}. For a particular sample $\{\xi_1,\ldots,\xi_N\}$, this figure depicts the feasible region
$$\left\lbrace x \in \mathbb{R}^2 \mid Q^{1-\alpha}(C(x,\xi)) \leq 0 \right\rbrace,$$
and its approximation based on the empirical quantile function; where $C(x,\xi) = \xi^Tx - 1$, $\xi \sim N(\mu,\Sigma)$, $\mu = \mathbf{0} \in \mathbb{R}^2$, and $\Sigma$ is the $2 \times 2$ identity matrix.

\begin{figure}[htbp]
\centering
\begin{subfigure}{0.5\textwidth}
  \centering
  \includegraphics[width=0.99\linewidth]{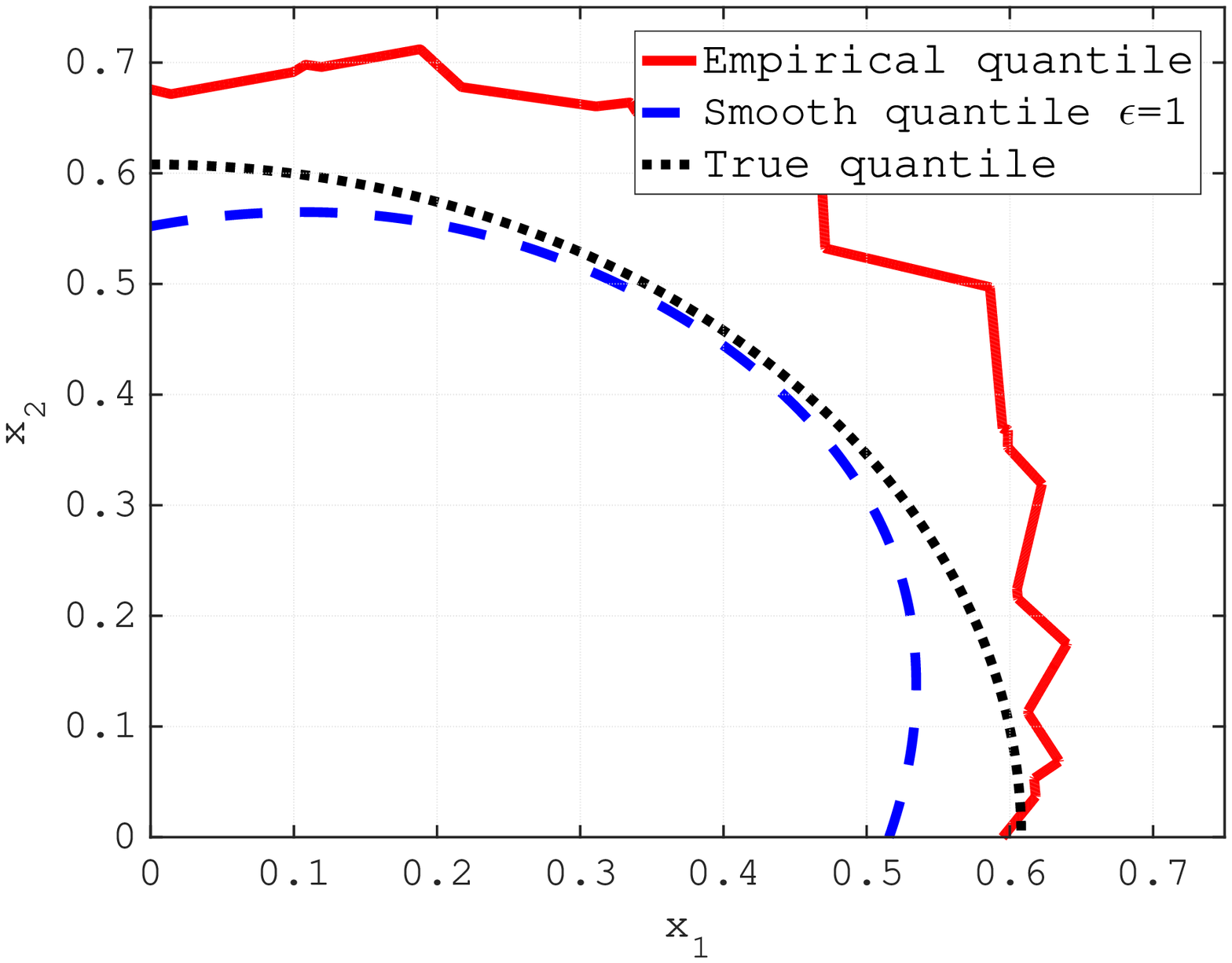}
  \caption{$N = 100$ observations.}
\end{subfigure}%
\begin{subfigure}{0.5\textwidth}
  \centering
  \includegraphics[width=0.99\linewidth]{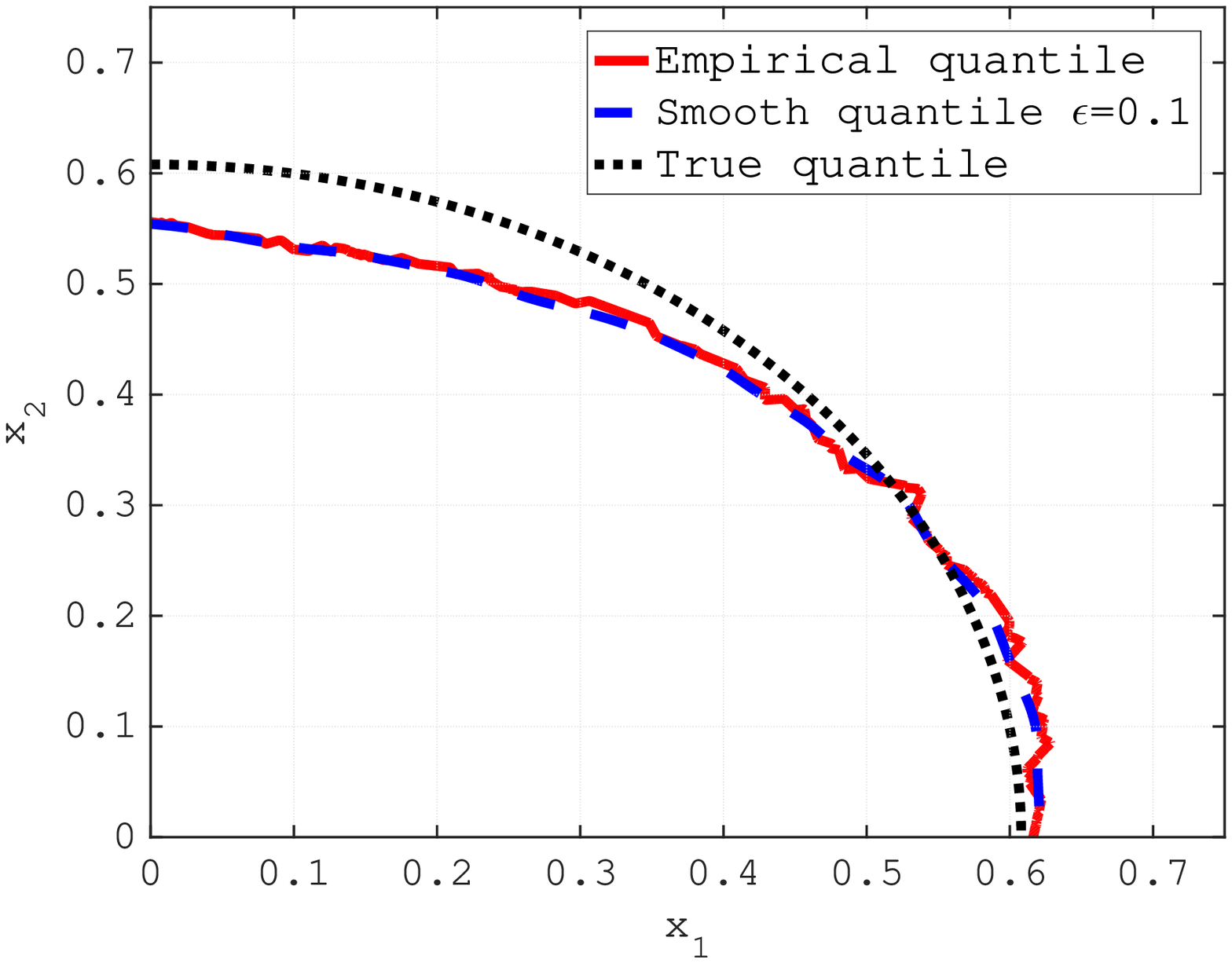}
  \caption{$N = 1,000$ observations.}
\end{subfigure}
\caption{Different quantile approximations ($\widetilde{Q}^{1-\alpha}(C^N(x))$ as a solid red line, and $Q_\epsilon^{1-\alpha}(C^N(x))$ as a dashed blue line) vs. true quantile $Q^{1-\alpha}(C(x,\xi))$.}
\label{fig: EmpTrue}
\end{figure}

Notice that even if the true feasible region is convex and has a smooth boundary (dotted curve), its approximation (solid curve) is non-smooth and introduces a lot of ``inward kinks'', which break the convexity of the feasible region. These ``inward kinks'' can potentially cause a local optimization algorithm to converge to inferior local minima more easily. Furthermore, as the number of scenarios increases, the number of non-smooth points increases, leading to more sub-optimal local minima.

In order to avoid the aforementioned problems, a smooth approximation of the quantile is obtained by estimating the cdf with a smooth function $\hat{F}$, and solving the equation $\hat{F}(t;x) = 1 - \alpha$, where $t$ is the approximation of $Q^{1 -\alpha}(C(x,\xi))$. Following an approach similar to \cite{GeleHoffKlop17} or \cite[Section 4.4.2]{ShapDentRusz09}, we approximate $F$ at a given point $x$ as
\begin{align} \label{eq: SmoothCDF}
F^N_\epsilon(t;x) = \frac{1}{N}\sum_{i = 1}^N \Gamma_\epsilon(C(x,\xi_i) - t),
\end{align}
where $\epsilon > 0$ is a parameter of the approximation,
\begin{align} \label{eq: SmoothIndicator}
\Gamma_\epsilon(y) = \left\{
	\begin{array}{lr}
		1,     & y \leq -\epsilon \\
  \gamma_\epsilon(y), & -\epsilon < y < \epsilon \\
		0,     &  y \geq \epsilon
	\end{array}
\right.
\end{align}
and $\gamma_\epsilon: [-\epsilon,\epsilon] \rightarrow [0,1]$ is a symmetric and strictly decreasing function such that it makes $\Gamma_\epsilon$ differentiable. With this choice of $\gamma_\epsilon$, $F^N_\epsilon(t;x)$ is a differentiable approximation of the empirical cdf, $\widetilde{F}^N(t;x)$.

Assuming that $C(x,\xi)$ is a continuous random variable for a fixed value of $x$, we can approximate the $(1-\alpha)$-quantile as the inverse of $F^N_\epsilon$ at the $1-\alpha$ level. Thus, we define our $\epsilon$-approximation to the $(1-\alpha)$-quantile at $x$ as the value $Q_\epsilon^{1-\alpha}$ such that the following equation holds
\begin{align}
\sum_{i = 1}^N \Gamma_\epsilon([C^N(x)]_i - Q_\epsilon^{1-\alpha}) = (1 - \alpha)N, \label{eq: SmoothVaR}
\end{align}
where $C^N(x) \in \mathbb{R}^N$ is a vector with elements $[C^N(x)]_i = C(x,\xi_i)$. However, the value $Q_\epsilon^{1-\alpha}$ is not unique if the equation $F^N_\epsilon(t;x) = 1-\alpha$ has more than one solution. In order to guarantee the existence of a unique solution for all possible choices of $\epsilon > 0$, we need to ensure that $[F^N_\epsilon]'(t;x) \neq 0$ when evaluated at the quantile of interest. The following proposition gives conditions under which $Q_\epsilon^{1-\alpha}$ is uniquely defined.

\begin{lemma} \label{prop: UniqueSolution}
Let $\gamma_\epsilon: [-\epsilon,\epsilon] \rightarrow [0,1]$ be a strictly decreasing function such that $\Gamma_\epsilon$ is differentiable, and let $(1-\alpha)N \notin \mathbb{Z}$. Then, equation \cref{eq: SmoothVaR} has a unique solution.
\end{lemma}

\begin{proof}
Let $\psi(Q) = \sum_{i = 1}^N \Gamma_\epsilon(C(x,\xi_i) - Q) - (1 - \alpha)N$. We have that $\psi'(Q) = -\sum_{i = 1}^N \Gamma'_\epsilon (C(x,\xi_i) - Q)$. Given that $\gamma_\epsilon$ is strictly decreasing and $\Gamma'_\epsilon(y) = 0$ for all $y \geq \epsilon$ and $y \leq -\epsilon$, we conclude that $\psi'(Q) \geq 0$ for all $Q$, which means that $\psi$ is an increasing function. If we prove that $\psi'(Q) \neq 0$ for all $Q$ such that $\psi(Q) = 0$, then we show that there is a unique solution of equation \cref{eq: SmoothVaR}.

Assume $\psi'(Q) = 0$, then $Q \leq C(x,\xi_i) - \epsilon$ or $Q \geq C(x,\xi_i) + \epsilon$ for all scenarios $i\in \{1,\ldots,N\}$. This means that $\sum_{i = 1}^N \Gamma_\epsilon (C(x,\xi_i)-Q) \in \mathbb{N}$. Because $(1 - \alpha)N \notin \mathbb{N}$, it means that $\psi(Q) \neq 0$, and therefore any $Q$ such that $\psi'(Q) = 0$ cannot be a solution.
\end{proof}

\begin{remark} \label{rem: NotIntegerB}
The assumption $(1-\alpha)N \notin \mathbb{Z}$ of \cref{prop: UniqueSolution} is mild because if it is not satisfied one can either increase $N$, perturb $\alpha$, or add a constant $b \in (0,1)$ to the left hand side of \cref{eq: SmoothVaR}. In fact, if a constant is added to the left-hand side, the impact will disappear as $N \rightarrow \infty$.
\end{remark}

\Cref{prop: UniqueSolution} shows that \cref{eq: SmoothVaR} indeed implicitly defines a function $Q^{1-\alpha}_\epsilon$ that maps a vector $C^N(x) \in \mathbb{R}^N$ to the root of \cref{eq: SmoothVaR}. In order to find solutions for \cref{eq: CCP}, we focus on solving the following problem
\begin{align} \label{eq: VaRFormulation}
\min_{x \in X} &\quad f(x) \\
\mathrm{s.t.} &\quad Q_\epsilon^{1-\alpha}(C^N(x)) \leq 0. \nonumber
\end{align}
The above formulation is a sample average approximation of the original problem \cref{eq: CCP}. Convergence and feasibility of this approximation are analyzed in \cref{sec: Feasibility}.

One could have approximated constraint \cref{eq: ProbaConst} with $F^N_\epsilon(0;x) \geq 1-\alpha$, where $F^N_\epsilon$ is given by \cref{eq: SmoothCDF}. However, as mentioned at the beginning of this section and as shown in the numerical results (\cref{sec: ReinsuranceSCCP}), working with the probabilistic constraint, or its estimate, instead of the quantile leads to more iterations or failure of the optimization algorithm. Therefore, in contrast with previous works that express the chance constraint using variations of \cref{eq: SmoothCDF}, we only use $F_\epsilon^N$ to obtain an approximation of the quantile and work directly with the quantile function.

While other options are possible, our choice of the smoothed indicator function, $\Gamma_\epsilon$, is motivated by a kernel estimation of the cdf. In \cite[p. 326]{AlemBolaGuil12}, the authors present the following estimator
\begin{align*}
\hat{F}^N_\epsilon(t;x) &= \int_{-\infty}^t \frac{1}{N\epsilon} \sum_{i = 1}^N \vartheta \left( \frac{u - C(x,\xi_i)}{\epsilon} \right) \, du = \frac{1}{N} \sum_{i = 1}^N \Theta \left( \frac{t - C(x,\xi_i)}{\epsilon} \right),
\end{align*}
where $\vartheta(\cdot)$ is a real-valued non-negative function, known as kernel function, that integrates to one over all of its domain (features that make it a probability density function) and is symmetric around 0. The function $\Theta(\cdot)$ is the cdf of $\vartheta(\cdot)$.

If $\vartheta(\cdot)$ has a finite support $[-1,1]$, then $\Theta \left( \frac{t-c(x,\xi)}{\epsilon}\right)$ has the structure of the $\Gamma_\epsilon$ function in \cref{eq: SmoothIndicator}. Furthermore, according to~\cite[p. 326]{Azza81}, if $\epsilon \rightarrow 0$ as $N \rightarrow \infty$, it can be proven that asymptotically
\begin{align}
\mathbb{E}\left[ \left( \hat{F}^N_\epsilon(t;x) - F(t;x) \right)^2 \right] \approx \frac{F(t;x)\left[ 1 - F(t;x)\right]}{N} -u(t) + \epsilon^4 v(t), \label{eq: MSESmoothKernel}
\end{align}
where
$$u(t) = f(t;x)\frac{\epsilon}{N}\left( 1 - \int_{-1}^1 \Gamma_1^2 (y) \, dy \right), \quad v(t) = \left( \frac{1}{2} f'(t;x) \int_{-1}^1 y^2 \Gamma_1' (y) \, dy \right)^2$$
The first two terms in \cref{eq: MSESmoothKernel} correspond to the asymptotic variance and the third term is the squared asymptotic bias \cite[p. 257]{AlemBolaGuil12}. This result indicates that large values of $\epsilon$ help reduce the variance of the estimator but increase its bias, while small values behave exactly opposite. Furthermore, if we compare \cref{eq: MSESmoothKernel} to the mean squared error (MSE) of the empirical distribution $\widetilde{F}^N(t;x)$, which equals $F(t;x)\left[ 1 - F(t;x)\right]/N$, we conclude that our estimation has a smaller MSE than the empirical distribution if $\epsilon$ is chosen optimally. Azzalini in \cite{Azza81} also shows that the properties of $\hat{F}^N_\epsilon$ extend to the quantile estimator, i.e., to the root of $\hat{F}^N_\epsilon(t;x) = 1 - \alpha$.

These observations regarding the behavior of the quantile approximation coincide with the findings of our numerical experiments. The solutions returned by the solver for large values of $\epsilon$ are generally feasible but sub-optimal due to the bias introduced by $\epsilon$. Large values of $\epsilon$ usually generate a conservative approximation of the quantile function; however, the solutions have less variance. We also find that solving the approximation with larger values of $\epsilon$ tends to require fewer iterations. Thus, there is a trade-off in the choice of $\epsilon$; small values of $\epsilon$ reduce bias, whereas larger values of $\epsilon$ reduce the variance and the computational effort. In \cref{sec: ChoosingEpsilon}, we suggest a method for selecting an appropriate value of $\epsilon$ for an instance.
Our proposal to find a positive $\epsilon$ that balances bias and variance contrasts with prior work \cite{CaoZava17,GeleHoffKlop17,HongYangZhan11,HuHongZhan13} that uses smoothing approximations which focused on driving $\epsilon$ to zero.

For the experiments in \cref{sec: NumRes} we use the quartic kernel \cite[p. 353]{ScotTapiThom77} to define function $\Gamma_\epsilon$. This gives the following $\gamma_\epsilon$;
\begin{align}
    \gamma_\epsilon(y) = \frac{15}{16}\left( -\frac{1}{5}\left(\frac{y}{\epsilon}\right)^5 + \frac{2}{3}\left(\frac{y}{\epsilon}\right)^3 - \left(\frac{y}{\epsilon}\right) + \frac{8}{15} \right), \label{eq: gamma_eps}
\end{align}
which makes $\Gamma_\epsilon(\cdot)$ a twice continuously differentiable function.

\begin{note}
For simplicity of notation, we assume in the following sections that the quantile of interest is always the $(1-\alpha)$-quantile and drop the superscript of the quantile approximation. Therefore, the $(1-\alpha)$-quantile $\epsilon$-approximation is denoted just as $Q_\epsilon$ for now onward.
\end{note}

\section{Convergence and feasibility of the approximation} \label{sec: Feasibility}

In this section we study how the solutions obtained by our method approximate those of the real problem \cref{eq: CCP}. Throughout the rest of the paper we make the following assumption.


\begin{assumption} \label{assu: ContinuousCDF}
For each $x \in X$, the random variable $C(x,\xi)$ has a continuous distribution.
\end{assumption}

Notice that, if $C(x,\xi)$ is not a continuous random variable, we can make \cref{assu: ContinuousCDF} hold for an
approximation of the problem by adding a continuous random variable $\bar{\xi}$ to $C(x,\xi)$ with expectation zero and small variance. This could be treated as a fixed approximation, or one could solve a sequence of approximations in which the variance of $\bar{\xi}$ is decreased to zero.

Let the feasible region of the original problem \cref{eq: CCP} be defined as 
$$X_\alpha = \left\lbrace x \in X \mid F(0;x) \geq 1 - \alpha \right\rbrace,$$
where $F(t;x) = \mathbb{P}(C(x,\xi) \leq t)$. Moreover, define $F_{\epsilon}(t;x)$, for $\epsilon > 0$, as
\begin{align}
F_\epsilon(t;x) = \int_{\mathbb{R}^s} \Gamma_{\epsilon} \left( C(x,\xi) - t \right) \,d\proba_\xi, \label{eq: CDFEps}
\end{align}
where $\Gamma_\epsilon$ is the smooth approximation of the indicator function given by \cref{eq: SmoothIndicator}.

Now, let $\{\xi_i\}_{i = 1}^N$ be a random sample of the variable $\xi$ of size $N$. As in \cref{sec: SmoothVaR}, we approximate $F_\epsilon$ with $F_\epsilon^N$, and denote the feasible region defined by the approximation as follows
\begin{align}
    X_{\epsilon,\delta}^{N,t} = \left\lbrace x \in X \mid F_\epsilon^{N}(-t;x) \geq 1 - \delta \right\rbrace, \quad \delta \in [0,1], \enskip t \in \mathbb{R}. \label{eq: ApproxFeasReg}
\end{align}

Notice that if $\delta = \alpha$ and $t = 0$, then $X_{\epsilon,\alpha}^{N,0}$ corresponds to the feasible region of \cref{eq: VaRFormulation}. We consider this approximation more generally with $\delta \leq \alpha$ and $t \geq 0$ as we will show how the parameters $t$ and $\delta$ in \cref{eq: ApproxFeasReg} can be chosen to increase the likelihood that the solution to the approximate model is feasible to the true model. First, notice that if $t = \epsilon$, then $F_\epsilon^N(-t;x) \leq \widetilde{F}^N(0;x)$. This implies that a point $x \in X_{\epsilon,\alpha}^{N,\epsilon}$ satisfies $\widetilde{F}^N(0;x) \geq 1 - \alpha$, so the solutions returned by our method are feasible for the empirical approximation of the feasible region. Moreover, the $\delta$ parameter can be used, as in \cite{LuedAhme08}, to ensure that the solutions obtained are feasible with high probability.

We start the analysis of \cref{eq: VaRFormulation} by showing asymptotic convergence of the optimal value of the approximation to the optimal value of the true problem. In particular, we analyze convergence of the approximation with respect to both the smoothing parameter and the sample size, whereas previous results, to the best of our knowledge, only consider convergence with respect to the smoothing parameter (\cite{GeleHoffKlop17, HongYangZhan11, HuHongZhan13}) or the sample size (\cite{LuedAhme08, PagnAhmeShap09}) separately. Then, we show complementary analysis of the approximation problem that indicates that, with appropriate choice of parameters, any feasible solution to the approximation problem will be feasible with high probability for sufficiently large sample size.

\subsection{Asymptotic convergence}

In the following theorems we state results analogous to Proposition 2.1 and 2.2 in \cite{PagnAhmeShap09}, thus the proofs we present are similar to those in that paper. We make the following assumptions.

\begin{assumption} \label{assu: Caratheodory}
$C(x,\xi)$ is a Carath\'eodory function, i.e., $C(x,\cdot)$ is measurable for every $x \in X$ and $C(\cdot,\xi)$ continuous for a.e. $\xi \in \Xi$.
\end{assumption}

\begin{assumption} \label{assu: ConvSeqToOptimum}
There exists a globally optimal solution of \cref{eq: CCP}, $\bar{x}$, such that for any $\delta > 0$ there is $x \in X$ such that $\|x - \bar{x}\| \leq \delta$ and $F(0;x) > 1 - \alpha$.
\end{assumption}

This assumption implies that there exists a sequence $\{x_k\}_{k=1}^\infty \subseteq X$ that converges to an optimal solution $\bar{x}$ such that $F(0;x_k) > 1 - \alpha$ for all $k \in \mathbb{N}$.




\begin{theorem} \label{teo: F_Conv}
Suppose that \cref{assu: ContinuousCDF,assu: Caratheodory} hold. Then, $F_\epsilon^N(0;x)$ converges to $F(0;x)$ uniformly on any compact set $U \subseteq X$ w.p.1, i.e.,
$$\sup_{x \in U} \left\lvert F_\epsilon^N(0;x) - F(0;x)\right\lvert \rightarrow 0, \quad\text{w.p.1}, \quad\text{as }N \rightarrow \infty \text{ and } \epsilon \rightarrow 0.$$
\end{theorem}

\begin{proof} We prove uniform convergence of $F_\epsilon^N(0;x)$ to $F(0;x)$ by first showing uniform convergence of $F_\epsilon^N(0;x)$ to $F_\epsilon(0;x)$ over $x \in U$ and $\epsilon \in [0,1]$ as $N \rightarrow \infty$, and then uniform convergence of $F_\epsilon(0;x)$ to $F(0;x)$ over $x \in U$ as $\epsilon \rightarrow 0$.

To prove uniform convergence of $F_\epsilon^N(0;x)$ to $F_\epsilon(0;x)$ in the compact set $U \subseteq X$ for all $\epsilon \in [0,1]$, we start by showing that $\Gamma_\epsilon(C(x,\xi))$ is continuous w.p.1 for all $x \in U$ and $\epsilon \in [0,1]$. Notice that $\Gamma_\epsilon(C(x,\xi))$ is continuous w.p.1 for all $x \in U$ and $\epsilon > 0$. This follows because $\Gamma_\epsilon(\cdot)$ is continuous and $C(x,\xi)$ is a Carath\'eodory function. It remains to prove that $\Gamma_\epsilon(C(x,\xi))$ is continuous in $x$ and $\epsilon$ at $\epsilon = 0$ w.p.1. Let $x$ and $\xi$ be such that $C(x,\xi)$ is continuous and strictly positive. Then, there exists $\epsilon_0 > 0$ and $\delta_0 >0$ such that $C(x,\xi) > \epsilon_0$ and $C(y,\xi) > \epsilon_0$ for all $y \in \mathcal{B}_{\delta_0}(x)$ (where $\mathcal{B}_{\delta_0}(x) = \{y \mid \|x-y\| \leq \delta_0 \}$). Thus, for all $\rho > 0$, $|\Gamma_0(C(x,\xi)) - \Gamma_\epsilon(C(y,\xi))| = 0 < \rho$ for all $0\leq \epsilon \leq \epsilon_0$ and $y \in \mathcal{B}_{\delta_0}(x)$. This shows that $\Gamma_0(C(x,\xi))$ is continuous when $C(x,\xi) > 0$. An analogous proof can be derived if $C(x,\xi)$ is strictly negative. Finally, because $C(x,\xi) \neq 0$ w.p.1 for all $x \in X$ from \cref{assu: ContinuousCDF}, we conclude that $\Gamma_\epsilon(C(x,\xi))$ is continuous w.p.1 for all $x \in U$ and $\epsilon \in [0,1]$.

Uniform convergence of $F_\epsilon^N(0;x)$ to $F_\epsilon(0;x)$ over $x \in U$ and $\epsilon \in [0,1]$ is then a result of applying Theorem 7.53 in \cite{ShapDentRusz09}. This is because function $\Gamma_\epsilon(C(x,\xi))$ is continuous w.p.1 for all $x \in U$ and $\epsilon \in [0,1]$, $|\Gamma_\epsilon(C(x,\xi))| \leq 1$ for all $x \in X$, and the sample is i.i.d.

Uniform convergence of $F_\epsilon(0;x)$ to $F(0;x)$ in the compact set $U \subseteq X$ as $\epsilon \rightarrow 0$ is a result of applying Theorem 4.91 in \cite{ShapDentRusz09}. This comes from \cref{assu: ContinuousCDF,assu: Caratheodory}.

Combining these two results, for every $\delta > 0$, we can find $\epsilon_{\delta} > 0$ and $N_\delta \in \mathbb{N}$ such that $|F_{\epsilon}(0;x) - F(0;x)| < \delta/2$ and $| F_\epsilon^N(0;x) - F_{\epsilon}(0;x)| < \delta/2$, w.p.1, for all $x \in U$, $\epsilon \leq \min\{1,\epsilon_{\delta}\}$, and $N \geq N_\delta$, which ensures that
$$\sup_{x \in U} \left\lvert F_\epsilon^N(0;x) - F(0;x)\right\lvert < \delta.$$
We can therefore conclude that $F_\epsilon^N(0;x)$ converges to $F(0;x)$ uniformly in the compact set $U$.
\end{proof}

We denote the sets of optimal solutions for problems \cref{eq: CCP} and \cref{eq: VaRFormulation} as $S$ and $S_\epsilon^N$, respectively. We also denote $v^*$ and $v_\epsilon^N$ the optimal values of \cref{eq: CCP} and \cref{eq: VaRFormulation}, respectively. We show next that, under some assumptions, $v^N_\epsilon$ and $S^N_\epsilon$ converge w.p.1 to their counterparts of the true problem as $N$ increases and $\epsilon$ decreases.

\begin{theorem} \label{teo: AsymConv}
Suppose that $X$ is compact, the function $f$ is continuous, and \cref{assu: ContinuousCDF,assu: Caratheodory,assu: ConvSeqToOptimum} hold. Then, $v_\epsilon^N \rightarrow v^*$ and $\mathbb{D}(S_\epsilon^N,S) \rightarrow 0$ w.p.1 as $N \rightarrow \infty$ and $\epsilon \rightarrow 0$.
\end{theorem}

\begin{proof}
Notice that for any given $N$ and $\epsilon$, a point $x$ in the feasible set of \cref{eq: VaRFormulation} satisfies $F_{\epsilon}^{N}(0;x) \geq 1 - \alpha$. Thus, we can show convergence of the feasible set of \cref{eq: VaRFormulation} to the feasible set of \cref{eq: CCP} by analyzing the function $F_{\epsilon}^{N}$.

Now, by \cref{assu: ConvSeqToOptimum}, the set $S$ is nonempty and there is $x \in X$ such that $F(0;x) > 1-\alpha$. Because $F_\epsilon^N(0;x)$ converges to $F(0;x)$, by \cref{teo: F_Conv}, there exists $\epsilon_0$ small enough and $N_0$ large enough such that $F_{\epsilon_0}^{N_0}(0;x) \geq 1 - \alpha$ w.p.1. Because $F_{\epsilon_0}^{N_0}$ is continuous in $x$ and $X$ is compact, the feasible set of the approximation problem \cref{eq: VaRFormulation} is compact as well, and hence $S_\epsilon^N$ is nonempty w.p.1 for all $N \geq N_0$ and $\epsilon \leq \epsilon_0$.


Let $\{N_k\}_{k = 1}^\infty \geq N_0$ and $\{\epsilon_k\}_{k = 1}^\infty \leq \epsilon_0$ be two sequences such that
$N_k \rightarrow \infty$ and $\epsilon_k \rightarrow 0$. Let $\hat{x}_k \in S_{\epsilon_k}^{N_k}$, i.e., $\hat{x}_k \in
X$, $F_{\epsilon_k}^{N_k}(0;\hat{x}_k) \geq 1 - \alpha$ and $v_{\epsilon_k}^{N_k} = f(\hat{x}_k)$. Let $\hat x\in X$ be
any cluster point of $\{x_k\}_{k = 1}^\infty$ and $\{x_l\}_{l = 1}^\infty$ be a subsequence converging to $\hat x$.  Because $F_{\epsilon_l}^{N_l}(0;x)$ is continuous and converges uniformly to $F(0;x)$ on $X$ w.p.1 by \cref{teo: F_Conv}, we have that  
$F(0;\hat{x}) = \lim_{l \rightarrow \infty} F_{\epsilon_l}^{N_l}(0;\hat{x}_l)$, w.p.1.
Hence, $F(0;\hat{x}) \geq 1 - \alpha$ and $\hat{x}$ is feasible for the true problem, and therefore $f(\hat{x}) \geq
v^*$. Also $f(\hat{x}_l) \rightarrow f(\hat{x})$ w.p.1 which means that $\lim_{l\rightarrow \infty} v_{\epsilon_l}^{N_l}
\geq v^*$ w.p.1.  Since this is true for any cluster point of $\{x_k\}_{k = 1}^\infty$ in the compact set $X$, we have
\begin{align}
    \liminf_{k\rightarrow \infty} v_{\epsilon_k}^{N_k} \geq v^*, \quad\text{w.p.1}. \label{eq: LimInfConv}
\end{align}


Now, by \cref{assu: ConvSeqToOptimum}, there exists an optimal solution $\bar x$ and a sequence $\{x_l\}_{l = 1}^\infty$ converging to $\bar x$ with $F(0;x_l)>1-\alpha$.  Since $F_{\epsilon_k}^{N_k}(0;{x}_l)$ converges to $F(0;x_l)$ w.p.1, there exist $K(l)$ such that $F_{\epsilon_k}^{N_k}(0;{x}_l)\geq 1-\alpha$ for every $k\geq K(l)$ and every $l$, w.p.1.  Without loss of generality we can assume that $K(l)<K(l+1)$ for every $l$ and define the sequence $\{\bar x_k\}_{k = K(1)}^\infty$ by setting $\bar x_k=x_l$ for all $k$ and $l$ with $K(l)\leq k< K(l+1)$.  We then have $F_{\epsilon_k}^{N_k}(0;\bar{x}_k)\geq 1-\alpha$, which implies $v_{\epsilon_k}^{N_k} \leq f(\bar x_k) $ for all $k\geq K(1)$. Since $f$ is continuous, we have that
\begin{align}
    \limsup_{k \rightarrow \infty} v_{\epsilon_k}^{N_k} \leq f(\bar{x}) = v^*, \quad\text{w.p.1}. \label{eq: LimSupConv}
\end{align}
It follows from \cref{eq: LimInfConv} and \cref{eq: LimSupConv} that $v_{\epsilon_k}^{N_k} \rightarrow v^*$ w.p.1.
To show that  $\mathbb{D}(S_\epsilon^N,S) \rightarrow 0$ w.p.1.\ we defer to Theorem 5.3 in \cite{ShapDentRusz09} which can easily be adapted to our context.
\end{proof}

Notice that \cref{teo: AsymConv} only applies to the globally optimal objective function and solutions. 




\subsection{Finite sample feasibility of approximation problem solutions} \label{sec: FiniteSampleFeas}

We now consider conditions under which an optimal solution for the approximation problem \cref{eq: VaRFormulation}, if it exists, is feasible for the true problem \cref{eq: CCP}. These results generalize those presented in Section 2.2 in \cite{LuedAhme08}, which considers the empirical cdf instead of a smooth approximation; therefore, some of our proofs follow from the ones presented there.

We will make use of Hoeffding's inequality.

\begin{theorem}[Hoeffding's Inequality \cite{Hoef63}]

Let $Y_1,\ldots,Y_N$ be independent random variables, with $\proba(Y_i \in [a_i,b_i]) = 1$, where $a_i \leq b_i$ for $i = 1,\ldots,N$. Then if $t>0$,
$$\proba \left\lbrace \sum_{1 = 1}^N \left( Y_i - \E[Y_i] \right) \geq tN \right\rbrace \leq \exp \left\lbrace -\frac{2N^2t^2}{\sum_{i=1}^N (b_i - a_i)^2} \right\rbrace.$$

\end{theorem}

\begin{assumption} \label{assu: PositiveM}
Let the scalars $t \in \mathbb{R}$ and $\delta \in [0,\alpha]$ be chosen such that $M := \inf_{x \in X} M_x > 0$, where $M_x = F(0;x) - F_\epsilon(-t;x) + (\alpha - \delta)$.
\end{assumption}

The above assumption can always be satisfied if we pick $t = \epsilon$ and $\delta < \alpha$. This claim can easily be observed from the definition of $F_\epsilon(-t;x)$ in \cref{eq: CDFEps}, which implies that $F_\epsilon(-\epsilon;x) \leq F(0;x)$ for all $x \in X$.

\begin{theorem}[Probabilistic feasibility guarantee] \label{teo: GeneralFeas}

Let $x \in X$ be such that $x \notin X_\alpha$, and suppose that \cref{assu: PositiveM} holds. Then
$$\proba \left( F^N_\epsilon(-t;x) \geq 1 - \delta \right) \leq \exp\left\lbrace -2N M_x^2 \right\rbrace.$$

\end{theorem}

\begin{proof}
Let $Y_i = \Gamma_\epsilon(C(x,\xi_i) + t)$ for all $i = 1,\ldots,N$, then $\mathbb{P}(Y_i \in [0,1]) = 1$ and $\E[Y_i] = F_\epsilon(-t;x)$. Given that $x$ is not a feasible solution for the original problem, $F(0;x) < 1 - \alpha$. We therefore have
\begin{align*}
\proba \left( F^N_\epsilon(-t;x) \geq 1 - \delta \right) &= \proba \left( F^N_\epsilon(-t;x) - F_\epsilon(-t;x) \geq 1 - \alpha + \alpha - \delta - F_\epsilon(-t;x) \right) \\
 &\leq \proba \left( F^N_\epsilon(-t;x) - F_\epsilon(-t;x) > F(0;x) - F_\epsilon(-t;x)  + (\alpha - \delta) \right) \\
 &= \proba \left( F^N_\epsilon(-t;x) - F_\epsilon(-t;x) > M_x \right) \\
 &= \proba \left( \sum_{i = 1}^N \left( Y_i - E[Y_i] \right) \geq M_x N\right) \\
 &\leq \exp\left\lbrace -\frac{2N^2 M_x^2}{\sum_{i = 1}^N \left( 1 - 0 \right)^2} \right\rbrace = \exp\left\lbrace -2N M_x^2 \right\rbrace,
\end{align*}
where the last inequality is obtained from Hoeffding's Inequality.
\end{proof}

\cref{teo: GeneralFeas} shows that if we shift the approximation of the cdf appropriately and/or decrease the risk level from $\alpha$ to $\delta \leq \alpha$, then a point $x$ that is feasible for the original problem \cref{eq: CCP} is also in $X_{\epsilon,\delta}^{N,t}$ with a probability that increases exponentially with the size of the sample $N$. Furthermore, under certain conditions, we can set $t = 0$ and $\delta = \alpha$ and still get the results stated above. This can be seen in the following corollary.

\begin{corollary} \label{prop: FeasOfOneX}

Let $x \in X$ be such that $ Y_x = C(x,\xi)$ is a continuous random variable with a strictly decreasing probability density function (pdf) $h_x(y)$ in the interval $(-\epsilon, \epsilon)$. If $F(0;x) < 1 - \alpha$, then
$$\proba \left( F^N_\epsilon(0;x) \geq 1 - \alpha \right) \leq \exp\left\lbrace -2N \beta_x^2 \right\rbrace,$$
where $\beta_x = F(0;x) - F_\epsilon(0;x) > 0$.

\end{corollary}

\begin{proof}

This follows from \cref{teo: GeneralFeas} for $t = 0$, $\delta = \alpha$, and noticing that $\beta_x = M_x = F(0;x) - F_\epsilon(0;x) > 0$ because
\begin{align*}
F(0;x) - F_\epsilon(0;x) &= \int_{\mathbb{R}^s} \mathbbm{1} \left( C(x,\xi) \right) \,d\proba_\xi - \int_{\mathbb{R}^s} \Gamma_{\epsilon} \left( C(x,\xi) \right) \,d\proba_\xi \\
    &= \int_{-\infty}^\infty \mathbbm{1} \left( y \right) h_x(y) \,dy - \int_{-\infty}^\infty \Gamma_\epsilon(y) h_x(y) \,dy \\
    &= \int_{-\epsilon}^0 \left( 1 - \Gamma_\epsilon(y) \right) h_x(y) \,dy - \int_{0}^\epsilon \Gamma_\epsilon(y) h_x(y) \,dy \\
 &= \int_{-\epsilon}^0 \Gamma_\epsilon(-y) h_x(y) \,dy - \int_{0}^\epsilon \Gamma_\epsilon(y) h_x(y) \,dy \\
 &= \int_{0}^\epsilon \Gamma_\epsilon(y) h_x(-y) \,dy - \int_{0}^\epsilon \Gamma_\epsilon(y) h_x(y) \,dy \\
 &= \int_{0}^\epsilon \Gamma_\epsilon(y) [h_x(-y) - h_x(y)] \,dy > 0.
\end{align*}

The fourth equality is the result of defining $\Gamma_\epsilon$ as the integral of a symmetric pdf, which implies that $1 - \Gamma_\epsilon(y) = \Gamma_\epsilon(-y)$. The last inequality is due to the fact that the pdf of $Y_x$ is strictly decreasing in the interval $(-\epsilon, \epsilon)$, which implies that $h_x(-y) > h_x(y)$ for $y \in (0,\epsilon)$, and due to $\Gamma_\epsilon(y) = \gamma_\epsilon(y) > 0$ for $y \in [0,\epsilon)$.
\end{proof}

One of the properties that follows from \cref{prop: FeasOfOneX} is that, for $\epsilon > 0$, our approximation is asymptotically conservative. In other words, for a fixed $\epsilon > 0$ and under a monotonicity assumption, a point $x$ that is feasible for the original problem \cref{eq: CCP} is also feasible for the approximation problem \cref{eq: VaRFormulation} with a probability that increases exponentially with the sample size $N$. It can be shown that for the vast majority of the common continuous distributions such as exponential, normal and Weibull, to mention a few, there exists a neighborhood of the solution such that the assumption in \cref{prop: FeasOfOneX} holds.

\begin{corollary}

Assume that $\Gamma_\epsilon$ is defined as in \cref{eq: SmoothIndicator}, where $\gamma_\epsilon(\cdot)$ is given by \cref{eq: gamma_eps}. Let $x \in X$ be such that $ Y_x = C(x,\xi)$ is a continuous random variable with pdf $h_x(y)$. If there exists $m_1, m_2 > 0$ such that $h_x(y) \geq -m_1y + h_x(0)$ for all $y \in (-\epsilon, 0]$, and $h_x(y) \leq -m_2y + h_x(0)$ for all $y \in [0, \epsilon)$. Then, we have that
$$\beta_x \geq \frac{(m_1+m_2)\epsilon^2}{28},$$
and thus, for all $x \notin X_\alpha$,
$$\proba \left( F^N_\epsilon(0;x) \geq 1 - \alpha \right) \leq \exp\left\lbrace -N (m_1 + m_2)^2 \epsilon^4/14 \right\rbrace.$$

\end{corollary}

\begin{proof}
From the definition of $\beta_x = F(0;x) - F_\epsilon(0;x)$, we get
\begin{align*}
F(0;x) - F_\epsilon(0;x) &= \int_{0}^\epsilon \Gamma_\epsilon(y) [h_x(-y) - h_x(y)] \,dy \\
 &\geq \int_{0}^\epsilon \Gamma_\epsilon(y) [(m_1 + m_2) y] \,dy \\
 &= \frac{(m_1+m_2)\epsilon^2}{28},
\end{align*}
where $\int_{0}^\epsilon \Gamma_\epsilon(y) y \,dy = \int_{0}^\epsilon \gamma_\epsilon(y) y \,dy = \epsilon^2/28$.
\end{proof}

We now consider conditions under which the entire feasible region of the approximation problem is a subset of the feasible region of problem \cref{eq: CCP}. For these results to hold we need to make additional assumptions. We begin with a preliminary result that holds under the assumption that the feasible region $X$ is finite.

\begin{theorem} \label{teo: FiniteFeas}

Suppose that $X$ is finite, and that \cref{assu: PositiveM} holds. Then
$$\proba \left( X_{\epsilon,\delta}^{N,t} \subseteq X_\alpha \right) \geq 1 - |X \setminus X_\alpha |\exp\left\lbrace -2NM^2 \right\rbrace.$$

\end{theorem}

\begin{proof}

Using the result from \cref{teo: GeneralFeas}, we get the following
\begin{align*}
\proba \left( X_{\epsilon,\delta}^{N,t} \not\subseteq X_\alpha \right) &= \proba \left(\exists x \in X_{\epsilon,\delta}^{N,t} \text{ such that } x \notin X_\alpha \right) \\
 &\leq \sum_{x \in X\setminus X_\alpha} \mathbb{P}\left(x \in X_{\epsilon,\delta}^{N,t} \right) \leq \sum_{x \in X\setminus X_\alpha} \exp \left( -2 N M_x^2 \right) \\
 &\leq |X \setminus X_\alpha |\exp\left\lbrace -2NM^2 \right\rbrace.
\end{align*}
\end{proof}

We now extend this result to the case in which $X$ is not finite. For this result, we make the following Lipschitz continuity assumption on $C$.

\begin{assumption} \label{assu: LipchitzCont}
There exists $L > 0$ such that
$$|C(x,\xi) - C(y,\xi)| \leq L \| x - y \|_{\infty}, \quad \forall x,y \in X, \text{ w.p.1}.$$
\end{assumption}

This assumption was made in \cite{LuedAhme08}. The proof stated below follows the same ideas as Theorem 10 in \cite{LuedAhme08}.

\begin{theorem} \label{teo: LipschitzFeas}

Assume that $X$ is bounded with diameter $D$ and \cref{assu: LipchitzCont} holds. Also, suppose that \cref{assu: PositiveM} holds with $t > 0$ and $\delta > 0$. Then, for any $\beta \in \left( 0, \delta \right]$, the following is true:
$$\proba \left( X_{\epsilon,(\delta-\beta)}^{N,2t} \subseteq X_\alpha \right) \geq 1 - \left \lceil 1/\beta \right \rceil \left \lceil 2LD/t \right \rceil^n \exp\left\lbrace -2NM^2 \right\rbrace.$$

\end{theorem}

\begin{proof}

Let $J = \lceil 1/\beta \rceil$. For $j = 1,\ldots,J-1$, define
$$X_j = \left\lbrace x \in X \mid \frac{j-1}{J} \leq F(0;x) < \frac{j}{J} \right\rbrace,$$
and let $X_J = \lbrace x \in X \mid (J-1)/J \leq F(0;x) \leq 1 \rbrace$. Following the arguments used to prove Theorem 10 in \cite{LuedAhme08}, we have that for each $j$ there exists a finite set $Z_j \subseteq X_j$ such that $|Z_j| \leq \lceil 2LD/t \rceil^n$ and that for all $x \in X_j$ there exists $z \in Z_{j}$ such that $\|x - z\|_\infty \leq t/L$.

Now define $Z = \cup_{j=1}^J Z_j$, and notice that $|Z| \leq \lceil 1/\beta \rceil \lceil 2LD/t \rceil^n$. Next, define $Z_{\alpha - \beta} = \{ x \in Z \mid F(0;x) \geq 1 - \alpha + \beta\}$ and
$$Z_{\delta-\beta}^{N,t} = \left\lbrace x \in Z \mid F_\epsilon^{N}(-t;x) \geq 1 - \delta + \beta \right\rbrace.$$
Since $Z$ is finite, we can apply \cref{teo: FiniteFeas} to obtain
\begin{align} \label{eq: }
    \proba \left( Z_{\delta - \beta}^{N,t} \subseteq Z_{\alpha-\beta} \right) \geq 1 - \left \lceil 1/\beta \right \rceil \left \lceil 2LD/t \right \rceil^n \exp\left\lbrace -2NM^2 \right\rbrace.
\end{align}

Consider now any $x \in X_{\epsilon,(\delta-\beta)}^{N,2t}$. Let $j \in \{1,\ldots,J\}$ be such that $x \in X_j$. By definition of $Z_j$, there exists $z \in Z_j$ such that $\|x - z\|_\infty \leq t/L$. By \cref{assu: LipchitzCont}, we know that $|C(x,\xi_i) - C(z,\xi_i)| \leq t$ w.p.1. Therefore, $C(z,\xi_i) + t \leq C(x,\xi_i) + 2t$, which implies that $\Gamma_{\epsilon}(C(z,\xi_i) + t) \geq \Gamma_{\epsilon}(C(x,\xi_i)+ 2t)$ w.p.1, and thus $F_\epsilon^{N}(-t;z) \geq F_\epsilon^{N}(-2t;x) \geq 1 - \delta + \beta$ w.p.1. We conclude that if $x \in X_{\epsilon,(\delta-\beta)}^{N,2t}$, then there exists $z \in Z_{\delta - \beta}^{N,t}$ w.p.1. Furthermore, by the definition of $X_j$ and because $Z_j \subseteq X_j$, we have $|F(0;x) - F(0;z)| \leq \beta$. If $Z_{\delta-\beta}^{N,t} \subseteq Z_{\alpha-\beta}$ then $F(0;z) \geq 1 - \alpha + \beta$, which implies that $F(0;x) \geq F(0;z) - \beta \geq 1 - \alpha$, and the result follows.
\end{proof}

\section{Single Chance-Constrained Problems} \label{sec: SCCP}

A single chance constrained problem is one in which $m = 1$, which means that $C(x,\xi) = \max_{j=1,\ldots,m}\{ c_j(x,\xi) \} = c(x,\xi)$. Defining $C^N(x) = c^N(x) = [c(x,\xi_1),\ldots,c(x,\xi_N)]^T$, problem \cref{eq: VaRFormulation} then becomes
\begin{subequations} \label{eq: SCCP}
\begin{align}
\min_{x \in X} &\quad f(x) \\
\text{s.t.} &\quad q(x) = Q_\epsilon(c^N(x)) \leq 0. \label{eq: VaRConstraint}
\end{align}
\end{subequations}

The advantage of the single-chance constrained case is that the approximated probabilistic constraint can be formulated in such a way that it results in a single smooth constraint (assuming $c$ is smooth in $x$) that can be handled by standard nonlinear optimization solvers and can be added to any nonlinear optimization problem. Furthermore, for a given $x$ and $\epsilon > 0$, we can use the implicit function theorem to obtain the gradient of $q$ (and provide it to the solver) as follows
\begin{align} \label{eq: ImplicitGradient}
\nabla q(x) &= \frac{ \sum_{i = 1}^N \Gamma_\epsilon'(c(x,\xi_i) - Q_\epsilon(c^N(x)))\nabla c(x,\xi_i)}{\sum_{j = 1}^N \Gamma_\epsilon'(c(x,\xi_j) - Q_\epsilon(c^N(x)))},
\end{align}
which is well defined as shown in \cref{prop: UniqueSolution}.

In fact, the following proposition shows that if we define $\gamma_\epsilon$ as \cref{eq: gamma_eps}, then function $Q_\epsilon(\cdot)$ has Lipschitz continuous gradients. Using the chain rule, this property can be extended to $q(x)$ as long as $c(x,\xi)$ has Lipschitz continuous gradients for all instances $\xi_i$.

\begin{proposition} \label{prop: LipschitzQ}
If $\gamma_\epsilon$ is defined as \cref{eq: gamma_eps} and $(1-\alpha)N \notin \mathbb{Z}$, then the function $Q_\epsilon: \mathbb{R}^N \rightarrow \mathbb{R}$ has bounded Lipschitz continuous gradients.
\end{proposition}

\begin{proof}
Let $z \in \mathbb{R}^N$. It follows from \cref{prop: UniqueSolution} and the implicit function theorem that
$$[\nabla Q_\epsilon(z)]_i = \frac{\Gamma_\epsilon' (z_i - Q_\epsilon(z))}{\sum_{j = 1}^N \Gamma_\epsilon'(z_j - Q_\epsilon(z))}.$$

We begin by showing that $\sum_{j = 1}^N \Gamma_\epsilon'(z_j - Q_\epsilon(z))$ is bounded away from zero. We prove this claim by contradiction. Assume that there exists $z \in \mathbb{R}^N$ such that $|z_j - Q_\epsilon(z)| \geq \epsilon$ for all $j \in \{1,\ldots,N\}$. Then, from \cref{eq: SmoothIndicator}, $\Gamma_\epsilon (z_j - Q_\epsilon(z)) \in \{0,1\}$ for all $j$. This implies that $\sum_{j = 1}^N \Gamma_\epsilon(z_j - Q_\epsilon(z)) \in \mathbb{N}$, which means that $\sum_{j = 1}^N \Gamma_\epsilon(z_j - Q_\epsilon(z)) \neq (1-\alpha)N$. Therefore, there must exists $i \in \{1,\ldots,N\}$ and $0 < \epsilon_0 < \epsilon$ such that $|z_i - Q_\epsilon(z)| \leq \epsilon_0$ and $\sum_{j = 1}^N \Gamma_\epsilon'(z_j - Q_\epsilon(z)) \leq \Gamma_\epsilon'(z_i - Q_\epsilon(z)) \leq \gamma_\epsilon'(\epsilon_0) < 0$, for all $z \in \mathbb{R}^N$.

More over, since $\sum_{j = 1}^N \Gamma_\epsilon'(z_j - Q_\epsilon(z)) \leq \Gamma'_\epsilon(z_i - Q_\epsilon(z)) \leq 0$ for all $i$, we conclude that $0 \leq [\nabla Q_\epsilon(z)]_i \leq 1$ for all $i$; this means that $\nabla Q_\epsilon(z)$ is bounded. Finally, because $\Gamma_\epsilon''$ exists and is bounded, then $\Gamma_\epsilon'$ is Lipschitz continuous. Given that the denominator is bounded above by $\gamma_\epsilon'(\epsilon_0)<0$, we can conclude that $Q_\epsilon(z)$ has Lipschitz continuous gradients.
\end{proof}

If we also assume that $f$ and $c$ are twice continuously differentiable, then we can provide the Hessian of constraint \cref{eq: VaRConstraint} to the solver as well. The Hessian is then computed as
\begin{align} \label{eq: SCCHessian}
\nabla^2 q(x) &= \sum_{i = 1}^N \left[\nabla Q_\epsilon \left(c^N(x)\right)\right]_i \left[ \nabla^2 c(x,\xi_i) \right] + \nabla c^N(x) \left[ \nabla^2 Q_\epsilon \left(c^N(x)\right) \right] \left[ \nabla c^N(x) \right]^T,
\end{align}
where $\left[\nabla Q_\epsilon (\cdot)\right]_i$ is the $i$th entry of the gradient of $Q_\epsilon(\cdot)$, and $\nabla^2 Q_\epsilon \left( \cdot \right)$ is the Hessian of the smooth approximation of the quantile with respect to each scenario, i.e., $\nabla^2 Q_\epsilon \left( \cdot \right) \in \mathbb{R}^{N \times N}$ is such that
\begin{align*}
    \left[\nabla^2 Q_\epsilon (z)\right]_{ii} &= \frac{w_i^2W' - 2w_iw_i'W + w_i'W^2}{W^3},\\
    \left[\nabla^2 Q_\epsilon (z)\right]_{ij} &= \frac{w_iw_jW' - W(w_iw_j' + w_i'w_j)}{W^3}, \quad i \neq j,
\end{align*}
where $w_\ell = \Gamma_\epsilon'(z_\ell - Q(z))$, $w_\ell' = \Gamma_\epsilon''(z_\ell - Q(z))$, $W = \sum_{\ell = 1}^N w_\ell$, and $W' = \sum_{\ell = 1}^N w_\ell'$.

In \cref{sec: NumResSCCP} we discuss the performance of this approximation. We give two examples in which the nonlinear constraint $q(x)$ and its derivatives are given directly to \texttt{Knitro}.

\section{Joint Chance-Constrained Problems} \label{sec: JCCP}

In this section, we consider problems where $c(x,\xi)$ is a vector-valued function from $\mathbb{R}^n \times \mathbb{R}^s$ to $\mathbb{R}^m$ for $m > 1$, and therefore $C(x,\xi) = \max_{j=1,\ldots,m}\{ c_j(x,\xi) \}$ is not smooth even when $c_j(x,\xi)$ are smooth for $j = 1,\ldots,m$.

Assuming that there exists a deterministic differentiable function, $g: \mathbb{R}^n \rightarrow \mathbb{R}^p$, such that the deterministic feasible set $X$ can be written as $X = \left\lbrace x \in \mathbb{R}^n \mid g(x) \leq 0 \right\rbrace$, we can rewrite the approximation of \cref{eq: VaRFormulation} as
\begin{subequations} \label{eq: SingleJCCP}
\begin{align}
\min_{x \in \mathbb{R}^n} &\quad f(x) \\
\text{s.t.} &\quad g(x) \leq 0, \\
 &\quad Q_\epsilon(C^N(x)) \leq 0. \label{eq: SingleVaRJCCP}
\end{align}
\end{subequations}
The discussion in this section can be extended to solve problems with deterministic equality constraints in the definition of the set $X$, but we consider only inequality constraints for a succinct exposition.

Notice that, because $C(x,\xi)$ is no longer smooth, the constraint \eqref{eq: SingleVaRJCCP} is not necessarily differentiable. To obtain a smooth reformulation of the quantile constraint \cref{eq: SingleVaRJCCP}, we introduce auxiliary variables $z_i$ that correspond to $\max_j\{c_j(x,\xi_i)\}$ and consider
\begin{subequations} \label{eq: JCCP}
\begin{align}
\min_{x \in \mathbb{R}^n, z \in \mathbb{R}^N} &\quad f(x) \\
\text{s.t.} &\quad g(x) \leq 0, \label{eq: Det_g}\\
 &\quad c(x,\xi_i) \leq z_i \pmb{1}_m, \quad\forall i = \{1,\ldots,N\} \label{eq: MaxConstraints}\\
 &\quad Q_\epsilon(z) \leq 0, \label{eq: VaRJCCPReformulation}
\end{align}
\end{subequations}
where $\pmb{1}_m$ is a vector of ones of size $m$.

\begin{proposition}
Let $x^*$ be a local minimum of \cref{eq: SingleJCCP}, then $(x^*,C^N(x^*))$ is a local minimum of \cref{eq: JCCP}.
\end{proposition}

\begin{proof}
Assume that $(x^*,C^N(x^*))$ is not a local minimum of \cref{eq: JCCP}. This means that there exists $(\hat{x},\hat{z})$ in a neighborhood of $(x^*,C^N(x^*))$ such that $f(\hat{x}) < f(x^*)$, $g(\hat{x}) \leq 0$, $c_j(\hat{x},\xi_i) \leq \hat{z}_i$ for all $i$ and $j$, and $Q_\epsilon(\hat{z}) \leq 0$; thus, $[C(\hat{x},\xi)]_i \leq \hat{z}_i$ for all $i$. Now, the function $Q_\epsilon$ is an increasing function with respect to $\hat{z}_i$ because $0 \leq [\nabla Q_\epsilon(\hat{z})]_i$, as shown in \cref{prop: LipschitzQ}. This means that $Q(C^N(\hat{x},\xi)) \leq Q(\hat{z}) \leq 0$. Putting everything together, we see that $\hat{x}$ is feasible for (5.1) and $f(\hat{x}) < f(x^*)$.  This argument holds for $(\hat{x},\hat{z})$  from an arbitrarily small neighborhood of $(x^*,C^N(x^*))$.  This is a contradiction because $x^*$ is assumed to be a local minimum of (5.1).
\end{proof}

We have shown that a local minimum of \cref{eq: SingleJCCP} is also a local minimum of \cref{eq: JCCP} if we define $z_i^* = [C^N(x^*)]_i$. However, a local minimum of \cref{eq: JCCP} is not necessarily a local minimum of \cref{eq: SingleJCCP} if $z_i^* \neq [C^N(x^*)]_i$. Thus, the downside of formulating the JCCP as the SCCP in \cref{eq: JCCP} is that a standard NLP solver might get stuck at a spurious local minimum. We therefore introduce a tailored algorithm that is designed to solve JCCPs.

\subsection{Exact-penalty formulation}

In order to solve problem \cref{eq: SingleJCCP}, we write an equivalent unconstrained optimization problem in which the constraints are handled by an $\ell_1$-penalty term. Our approach is motivated by the S$\ell_1$QP method for regular NLPs \cite{NoceWrig06}. We define the penalty function
\begin{align}
\phi_\pi(x) = f(x) + \pi\left( \left\| \left[ g(x) \right]^+ \right\|_1 + \left[ Q_\epsilon(C^N(x)) \right]^+ \right), \label{eq: Penalization}
\end{align}
where $\pi > 0$ is a penalty parameter and $[x]^+_i = \max\{0,x_i\}$. If $x^*$ is a minimizer of $\phi_\pi$ for $\pi > 0$ and $\| [ g(x) ]^+ \|_1 + [ Q_\epsilon(C^N(x)) ]^+ = 0$, then $x^*$ solves \cref{eq: SingleJCCP} \cite[p. 299]{HiriLema93}. The following theorem, stated in \cite[Theorem 2.1]{CurtOver12}, considers the converse situation and shows that there exists $\pi > 0$ such that we can use the minimization of \eqref{eq: Penalization} to obtain a solution for \cref{eq: SingleJCCP}.




\begin{theorem} \label{teo: ExactPenalty}
Suppose $f$, $g$, $c$, and $Q_\epsilon$ are locally Lipschitz on $\mathbb{R}^n$ and let $x^*$ be a local minimizer of \cref{eq: SingleJCCP}. Moreover, suppose that the Mangasarian-Fromowitz condition holds for problem \cref{eq: JCCP} at $x^*$ and $z^*=C^N(x^*)$. Then, there exists $\pi^* > 0$ and $\delta > 0$ such that for all $x \in \mathbb{B}_{\delta}(x^*)$,
$$\phi_\pi(x) \geq \phi_\pi(x^*) \quad\text{for all} \quad\pi \geq \pi^*.$$
\end{theorem}

To minimize function \cref{eq: Penalization}, we propose a trust-region method that generates steps from the minimization of a piecewise quadratic model of $\phi_\pi$ in every iteration. By making sure that the model approximates $\phi_\pi$ to first-order, we can invoke existing convergence results that show that the trust-region method converges to a stationary point of $\phi_\pi$. A piecewise quadratic model with such characteristics is defined in the following proposition.

\begin{proposition} \label{prop: FirstOrderModel}

Let $f$ and $g$ be differentiable with Lipschitz continuous gradients. Also, given a sample $\{\xi_1,\ldots,\xi_N\}$, let $c(x,\xi_i)$ be differentiable with respect to $x$ for all $i \in \{1,\ldots,N\}$ and assume there exists $L > 0$ and $L' > 0$ such that
\begin{align}
    \left\lvert c(x,\xi_i) - c(y,\xi_i) \right\rvert &\leq L \| x - y \|, \quad \forall x,y \in X, \enskip \forall i \in \{1,\ldots,N\}, \label{eq: LipschitzContinuity} \\
    \left\lvert \nabla c(x,\xi_i) - \nabla c(y,\xi_i) \right\rvert &\leq L' \| x - y \|, \quad \forall x,y \in X, \enskip \forall i \in \{1,\ldots,N\}. \label{eq: GradLipsCont_c}
\end{align}

Given a point $x \in \mathbb{R}^n$ and a direction $d \in \mathbb{R}^n$, consider the function
\begin{align}\label{eq: linear model}
m(x,H;d) &= f(x) + \nabla f(x)^T d + \frac{1}{2}d^T H d \nonumber \\
  &\qquad + \pi \left( \left\| \left[ g(x) + \nabla g(x)^Td \right]^+ \right\|_1 + \left[ \widetilde{Q}_{\epsilon,C}(x;d) \right]^+ \right),
\end{align}
where $H \in \mathbb{R}^{n\times n}$ is a symmetric matrix, and
\begin{align*}
\widetilde{Q}_{\epsilon,C}(x;d) &= \widetilde{Q}_\epsilon(C^N(x);\widetilde{C}^N(x;d)-C^N(x)), \\
\widetilde{Q}_\epsilon(z;p) &= Q_\epsilon(z) + \nabla Q_\epsilon(z)^T p, \\
[\widetilde{C}^N]_i(x;d) &= \max_j\{c_j(x,\xi_i) + \nabla c_j(x,\xi_i)^T d\}.
\end{align*}
Then, $m(x,H;d)$ is a first-order model of $\phi_\pi(x+d)$ in the sense that
$$\left\lvert \phi_\pi(x + d) - m(x,H;d)  \right\rvert = \mathcal{O}(\| d \|^2),$$
for all $d, x \in \mathbb{R}^n$ and all $H$ with $\|H\|\leq L_H$, for fixed $L_H$.

\end{proposition}

\begin{proof}

We only prove that $|Q_\epsilon(C^N(x+d)) - \widetilde{Q}_{\epsilon,C}(x;d)| = \mathcal{O}(\| d \|^2)$ since the rest follows from the Lipschitz continuity of $\nabla f$ and $\nabla g$, and boundedness of $H$.

For brevity in the proof, we will use the following notation $C_y = C^N(y)$ and $\widetilde{C}_{(x;d)} = \widetilde{C}^N(x;d)$. We have that,
\begin{align*}
& \left\lvert Q_\epsilon ( C_{x+d} ) - \widetilde{Q}_{\epsilon,C}(x;d) \right\rvert = \left\lvert Q_\epsilon(C_{x+d}) - \widetilde{Q}_\epsilon(C_x;\widetilde{C}_{(x;d)}-C_x) \right\rvert \\
&\enskip \leq \left\lvert Q_\epsilon(C_{x+d}) - \widetilde{Q}_\epsilon(C_x;C_{x+d}-C_x) \right\rvert + \left\lvert \widetilde{Q}_\epsilon(C_x;C_{x+d}-C_x) - \widetilde{Q}_\epsilon(C_x;\widetilde{C}_{(x;d)}-C_x) \right\rvert.
\end{align*}
Analyzing the first term, from Taylor's Theorem we know that there exists $\zeta_1 \in (0,1)$ such that
\begin{align*}
&\left\lvert Q_\epsilon(C_{x+d}) - \widetilde{Q}_\epsilon(C_x;C_{x+d} - C_x) \right\rvert \\
&\qquad = \left\lvert Q_\epsilon(C_{x+d}) - Q_\epsilon(C_x) - \nabla Q_\epsilon(C_x)^T \left[ C_{x+d}-C_x \right] \right\rvert \\
&\qquad = \left\lvert \nabla Q_\epsilon(C_x + \zeta_1 [ C_{x+d}-C_x ])^T \left[ C_{x+d}-C_x \right] - \nabla Q_\epsilon(C_x)^T \left[ C_{x+d}-C_x \right] \right\rvert \\
&\qquad \leq \| \nabla Q_\epsilon(C_x + \zeta_1 [ C_{x+d}-C_x ]) - \nabla Q_\epsilon(C_x) \| \| C_{x+d}-C_x \|  \\
&\qquad \leq L_Q \zeta_1 \| C_{x+d}-C_x \|^2 \leq L_Q L^2 \zeta_1 \| d \|^2 \leq L_Q L^2 \| d \|^2,
\end{align*}
where the two last inequalities come from \cref{prop: LipschitzQ} and \cref{eq: LipschitzContinuity} (we denote $L_Q$ as the Lipschitz constant of $\nabla Q_\epsilon$).

For the second term we have
\begin{align*}
&\left\lvert \widetilde{Q}_\epsilon(C_x;C_{x+d}-C_x) - \widetilde{Q}_\epsilon(C_x;\widetilde{C}_{(x;d)}-C_x) \right\rvert \\
&\qquad = \left\lvert Q_\epsilon(C_x) + \nabla Q_\epsilon(C_x)^T \left[ C_{x+d}-C_x \right] - Q_\epsilon(C_x) - \nabla Q_\epsilon(C_x)^T \left[ \widetilde{C}_{(x;d)}-C_x \right] \right\rvert \\
&\qquad = \left\lvert \nabla Q_\epsilon(C_x)^T \left[ C_{x+d} - \widetilde{C}_{(x;d)} \right] \right\rvert \leq \left\| \nabla Q_\epsilon(C_x) \right\| \left\| C_{x+d} - \widetilde{C}_{(x;d)} \right\|.
\end{align*}
Because $\nabla Q_\epsilon$ by \cref{prop: LipschitzQ} is bounded, it only remains to show that $\| C_{x+d} - \widetilde{C}_{(x;d)} \| = \mathcal{O}(\| d \|^2)$.

To prove this, first notice that
$$c_j(x,\xi_i) + \nabla c_j(x,\xi_i)^Td \leq [\widetilde{C}_{(x;d)}]_i, \quad \forall j \in \{ 1, \ldots, m \}, i \in \{1,\ldots,N\}.$$
Therefore, from Taylor's Theorem we have that there exists $\zeta_2 \in (0,1)$ such that
\begin{align} 
&[C_{x+d}]_i - [\widetilde{C}_{(x;d)}]_i = c_{j_1^i}(x+d,\xi_i) - [\widetilde{C}_{(x;d)}]_i \nonumber \\
&\qquad \leq c_{j_1^i}(x+d,\xi_i) - c_{j_1^i}(x,\xi_i) - \nabla c_{j_1^i}(x,\xi_i)^T d \nonumber \\
&\qquad = \nabla c_{j_1^i}(x + \zeta_2 d, \xi_i)^T d - \nabla c_{j_1^i}(x,\xi_i)^T d \nonumber\\
&\qquad \leq \| \nabla c_{j_1^i}(x + \zeta_2 d, \xi_i) - \nabla c_{j_1^i}(x, \xi_i) \| \| d \| \leq L' \zeta_2 \| d \|^2 \leq L' \| d \|^2, \label{eq: CIneq_1}
\end{align}
where $j_1^i \in \argmax_j \{c_j(x+d,\xi_i) \}$. Analogous to \cref{eq: CIneq_1}, we get
\begin{align}
[\widetilde{C}_{(x;d)}]_i - [C_{x+d}]_i \leq c_{j_2^i}(x,\xi_i) + \nabla c_{j_2^i}(x,\xi_i)^Td - c_{j_2^i}(x+d,\xi_i) \leq L' \| d \|^2 , \label{eq: CIneq_2}
\end{align}
where $j_2^i \in \argmax_j \{ c_j(x,\xi_i) + \nabla c_j(x,\xi_i)^T d \}$. It follows from \cref{eq: CIneq_1} and \cref{eq: CIneq_2} that $\| C_{x+d} - \widetilde{C}_{(x;d)} \| \leq \sqrt{N} L' \| d \|^2$.
\end{proof}

Now that model \cref{eq: linear model} has been established, we proceed to give an outline of the algorithm used to minimize $\phi_\pi$. 

\subsection{Trust-Region Algorithm}

In this section we present an S$\ell_1$QP-type trust-region algorithm to minimize \cref{eq: Penalization}. This algorithm seeks to obtain a first-order stationary point of $\phi_\pi$ by solving a sequence of quadratic programs (QP). At each iteration $k$, we compute a step $d^k$ by minimizing the model $m(x^k,H^k;d)$ within a radius $\Delta_k$ for a given $H_k$. This problem is equivalent to
\begin{align}
\min &\quad f(x^k) + \nabla f(x^k)^T d + \frac{1}{2}d^T H^k d + \pi \left[ t^T \pmb{1}_p + w \right] \nonumber \\
\text{s.t.} &\quad \nabla g(x^k)^Td + g(x^k) \leq t, \nonumber \\
 &\quad \nabla c(x^k,\xi_i)^Td + c(x^k,\xi_i) \leq z_i \pmb{1}_m, \qquad i = \{1,\ldots,N\}, \label{eq: QPApprox}\\
 &\quad \nabla Q_\epsilon(C^N(x^k))^T(z-C^N(x^k)) + Q_\epsilon(C^N(x^k)) \leq w, \nonumber \\
 &\quad t, w \geq 0, \|d\|_\infty \leq \Delta^k. \nonumber
\end{align}
We describe the choice of $H^k$ in \cref{sec: Lagrangian}.

A step $d^k$, obtained from \cref{eq: QPApprox}, is accepted if it results in sufficient decrease of $\phi_\pi$, i.e., the step $d^k$ is taken only if the value $\phi_\pi(x^k + d^k)$ is sufficiently smaller than $\phi_\pi(x^k)$. We then proceed with iteration $k+1$, where now $x^{k+1} = x^k + d^k$. \Cref{alg: JCCPSolver} states the full algorithm.

The following theorem follows from \cite[Theorem 11.2.5]{ConnGoulToin00}. It states that any limit point $x^*$ of the sequence $\{x^k\}$ generated by \cref{alg: JCCPSolver} is a stationary point of function $\phi_\pi$. Thus, we conclude that $x^*$ is a local minimum of \cref{eq: SingleJCCP} if $x^*$ is a local minimizer of $\phi_\pi$, $g(x^*) \leq 0$, and $Q_\epsilon\left(C^N(x^*)\right) \leq 0$.

\begin{theorem} \label{teo: ConvTR}
Suppose that \cref{alg: JCCPSolver} does not terminate at a stationary point of $\phi_\pi$ in \cref{alg_step: dStep}, and that $x^*$ is a limit point of the sequence $\{x^k\}$ generated by \cref{alg: JCCPSolver}. Let the assumptions from \cref{prop: FirstOrderModel} hold and assume that, at each iteration, problem \cref{eq: QPApprox} is solved exactly. Then $x^*$ is a first-order critical point of $\phi_\pi$.
\end{theorem}

To apply \cite[Theorem 11.2.5]{ConnGoulToin00}, it is necessary that $m(x,H;d)$ is a first-order model, which was shown in \cref{prop: FirstOrderModel}. The condition that problem \cref{eq: QPApprox} is solved exactly can be relaxed as long as some Cauchy decrease condition holds (refer to \cite[Theorem 11.2.5]{ConnGoulToin00}).

\begin{algorithm}[htbp]
    \textbf{Inputs:} $\pi > 0$ (penalty parameter); $\hat{\Delta} > 0$, $\Delta_0 \in (0,\hat{\Delta})$, $\eta \in (0,1)$, $\tau_1 \in (0,1)$, and $\tau_2 > 1$ such that $1/\tau_2 \leq \tau_1$ (trust region parameters); $x_0$ (initial point)
	\begin{algorithmic}[1]
        \FOR{$k = 0,1,2,\ldots$}
		\STATE Set $C(x^k,\xi_i) = \max_{j = 1,\ldots,m}\{c_j(x^k,\xi_i) \}$ for all scenarios, compute $Q_\epsilon(C^N(x^k))$ and $\nabla Q_\epsilon(C^N(x^k))$, and choose $H^k$.
		\STATE Obtain $d^k$ by solving \cref{eq: QPApprox} (\textbf{if} $d^k = 0$ \textbf{stop}, stationary point reached). \label{alg_step: dStep}
		\STATE Compute the ratio $\rho^k =  \frac{\phi_\pi(x^k) - \phi_\pi(x^k + d^k)}{m(x^k,H^k;0) - m(x^k,H^k;d^k)}.$
		\IF {$\rho^k < \eta$}
    		\STATE $\Delta^{k+1} = \tau_1 \min\{\Delta^k, \|d^k\|_\infty\}$; $x^{k+1} = x^k$
		\ELSIF {$\rho^k \geq \eta$ and $\|d^k \|_\infty = \Delta^k$}
        	\STATE $\Delta^{k+1} = \min\{\tau_2\Delta^k,\hat{\Delta}\}$; $x^{k+1} = x^k + d^k$
    	\ELSE
    		\STATE $\Delta^{k+1} = \Delta^k$; $x^{k+1} = x^k + d^k$
    	\ENDIF
    	\STATE $k = k+1$
    	\ENDFOR
	\end{algorithmic}
	\caption{S$\ell_1$QP trust-region algorithm to solve JCCP}\label{alg: JCCPSolver}
\end{algorithm}

\subsection{Choice of Hessian} \label{sec: Lagrangian}

In this section we propose a choice of the Hessian $H^k$ in \cref{eq: QPApprox} that aims to attain fast local convergence.
The key observation is a close relationship between the non-smooth formulation \cref{eq: SingleJCCP} and the smooth optimization problem
\begin{align}
    \min_{x \in \mathbb{R}^n} &\quad f(x) \nonumber\\
    \text{s.t.} &\quad g(x) \leq 0, \label{eq: SmallJCCP} \\
        & \quad\overline{q}(x) := Q_\epsilon(\overline{C}^N(x)) \leq 0, \nonumber
\end{align}
where $\overline{C}^N(x)$ is defined below in \cref{lem: KKTPoint}.

Recall that, if $x^*$ is a local minimizer of \cref{eq: SingleJCCP}, then $x^*$ and $z^* = C^N(x^*)$ is a KKT point of \cref{eq: JCCP} under some constraint qualification (see \cref{teo: ExactPenalty}).
The following lemma shows that $x^*$ is also a KKT point of \cref{eq: SmallJCCP}.

\begin{lemma} \label{lem: KKTPoint}
Let $(x^*,z^*,\nu^*,\mu^*,\lambda^*)$ be a KKT point of \cref{eq: JCCP} with $z^* = C^N(x^*)$, where $\nu^*$, $\mu^*$ and $\lambda^*$ are the multipliers associated to constraints \cref{eq: Det_g}, \cref{eq: MaxConstraints} and \cref{eq: VaRJCCPReformulation}, respectively. For each $i \in \{1,\ldots,N\}$, set
\begin{align}
   [\bar{\mu}_i]_j = \frac{[\mu^*_i]_j}{ \lambda^* \left[ \nabla Q_\epsilon(C^N(x^*)) \right]_i}, &\quad\forall \enskip j \in \{1,\ldots,m\}, \label{eq: muBar}
\end{align}
if $\lambda^* \neq 0$ and $\left[ \nabla Q_\epsilon(C^N(x^*)) \right]_i \neq 0$. Else, select one $j$ such that $c_j(x^*,\xi_i) = C(x^*,\xi_i)$ and define $[\bar{\mu}_i]_j = 1$ and $[\bar{\mu}_i]_\ell = 0$ if $\ell \neq j$. Also, define the following function
\begin{align}
   [\overline{C}^N(x)]_i = \bar{\mu}_i^Tc(x,\xi_i), &\quad\forall \enskip i \in \{1,\ldots,N\}. \label{eq: CBar}
\end{align}
Then, $(x^*,\nu^*,\lambda^*)$ is a KKT point of \cref{eq: SmallJCCP}.
\end{lemma}

\begin{proof}

Because $(x^*,z^*,\nu^*,\mu^*,\lambda^*)$ is a KKT point of \cref{eq: JCCP}, the following conditions hold
\begin{subequations}
\begin{align}
    \nabla f(x^*) + \nabla g(x^*)\nu^* + \sum_{i=1}^N \nabla c(x^*,\xi_i)\mu^*_i &= 0, \label{eq: GradLagX}\\
    \lambda^*\left[ \nabla Q_\epsilon(z^*) \right]_i - \mu^{*T}_i \pmb{1}_m &= 0, \quad\forall \enskip i \in \{1,\ldots,N\}, \label{eq: GradLagZ}\\
    c(x^*,\xi_i) - z^*_i\pmb{1}_m &\leq 0, \quad\forall \enskip i \in \{1,\ldots,N\}, \label{eq: FeasibilityZ} \\
    Q_\epsilon(z^*) &\leq 0, \label{eq: FeasibilityVaR}\\
    \mu_i^T(c(x^*,\xi_i) - z^*_i\pmb{1}_m) &= 0, \quad\forall \enskip i \in \{1,\ldots,N\}, \label{eq: Complementarity} \\
    \lambda^* Q_\epsilon(z^*) &= 0, \quad\forall \enskip i \in \{1,\ldots,N\}, \label{eq: ComplQ} \\
    \lambda^* \geq 0, \enskip \mu_i^* &\geq 0, \quad\forall \enskip i \in \{1,\ldots,N\}, \label{eq: PositiveMult}
\end{align}
\end{subequations}
where $\nabla g$ and $\nabla c$ correspond to the transpose of the Jacobians of $g$ and $c$, respectively. Notice that the KKT conditions of \cref{eq: JCCP} associated with $g$ have been left out because they trivially extend to the KKT conditions of \cref{eq: SmallJCCP}.

If $\lambda^* \neq 0$ and $\left[ \nabla Q_\epsilon(C^N(x^*)) \right]_i \neq 0$, then
\begin{align}
    \left[\overline{C}^N(x^*)\right]_i &\stackrel{\cref{eq: CBar}}{=} \bar{\mu}_i^Tc(x^*,\xi_i) \stackrel{\cref{eq: muBar}}{=} \sum_{j = 1}^m \left( \frac{[\mu^*_i]_j}{ \lambda^* \left[ \nabla Q_\epsilon(C^N(x^*)) \right]_i} \right)c_j(x^*,\xi_i) \label{eq: CBarEqualZ}\\
    &\stackrel{\cref{eq: Complementarity}}{=} z_i^* \sum_{j = 1}^m \left( \frac{[\mu^*_i]_j}{ \lambda^* \left[ \nabla Q_\epsilon(z^*) \right]_i} \right) \stackrel{\cref{eq: GradLagZ}}{=} z_i^*. \nonumber
\end{align}
If $\lambda^* = 0$ or $\left[ \nabla Q_\epsilon(C^N(x^*)) \right]_i = 0$, it follows from the definition of $\bar{\mu}_i$, that $[\overline{C}^N(x^*)]_i = \bar{\mu}_i^Tc(x^*,\xi_i) = z_i^*$. Thus, feasibility, $Q_\epsilon(\overline{C}^N(x^*)) = Q_\epsilon(z^*) \leq 0$, and complementarity, $\lambda^* Q_\epsilon(z^*) = \lambda^* Q_\epsilon(\overline{C}^N(x^*))=0$, hold.

Now, note that the gradient of the Lagrangian of \cref{eq: SmallJCCP} is given by
\begin{align}
  \nabla \overline{\mathcal{L}}(x,\nu,\lambda) &= \nabla f(x) + \nabla g(x)\nu + \lambda \sum_{i=1}^N \left[ \nabla Q_\epsilon(\overline{C}^N(x)) \right]_i \nabla c(x,\xi_i)\bar{\mu}_i. \label{eq: GradJCCP}
\end{align}
Then,
\begin{align*}
    \nabla \overline{\mathcal{L}}(x^*,\nu^*,\lambda^*) &\stackrel[\cref{eq: GradJCCP}]{\cref{eq: CBarEqualZ}}{=} \nabla f(x^*) + \nabla g(x^*)\nu^* + \lambda^* \sum_{i=1}^N \left[ \nabla Q_\epsilon(\overline{C}^N(x^*)) \right]_i \nabla c(x^*,\xi_i)\bar{\mu}^*_i \\
    &\stackrel[\cref{eq: GradLagZ}]{\cref{eq: muBar}}{=} \nabla f(x^*) + \nabla g(x^*)\nu^* + \sum_{i=1}^N \nabla c(x^*,\xi_i)\mu^*_i \stackrel{\cref{eq: GradLagX}}{=} 0.
\end{align*}
Thus, $(x^*,\nu^*,\lambda^*)$ is a KKT point of \cref{eq: SmallJCCP}.
\end{proof}

\Cref{lem: KKTPoint} shows that we can use the norm of \cref{eq: GradJCCP} in the termination criterion for our algorithm. Also, from complementary slackness and feasibility, we have that most of the entries of the vector $\mu_i^*$ are zero. The reason is that only the $j$th entries of $\mu_i^*$ such that $c_j(x^*,\xi_i) = C(x^*,\xi_i)$ can be nonzero. This means that the positive entries of $\mu_i^*$ indicate which entries of vector $c(x^*,\xi_i)$ correspond to the maximum $C(x^*,\xi_i)$. Therefore, $\overline{C}^N(\cdot)$ can be interpreted as a smooth function that coincides with the maximum $C^N(\cdot)$ at the solution $x^*$.

Suppose we apply the basic SQP algorithm \cite{NoceWrig06} to \cref{eq: SmallJCCP} and compute steps $d^k$ from the solution of the QP
\begin{align}
\min &\quad f(x^k) + \nabla f(x^k)^T d + \frac{1}{2}d^T {H}^k d\nonumber\\
\text{s.t.} &\quad \nabla g(x^k)^Td + g(x^k) \leq 0, \label{eq: SQP}\\
 &\quad  \nabla \overline{q}(x^k)^Td + \overline{q}(x^k) \leq 0\nonumber
\end{align}
to update the iterate $x^{k+1} = x^k +d^k$.
If ${H}^k$ is chosen as the Hessian of the Lagrangian function $\overline{\mathcal{L}}$ of \cref{eq: SmallJCCP} at the current iterate $(x^k,\nu^k,\lambda^k)$, i.e.,
\begin{align}
    H^k = \nabla^2 f(x^k) &+ \sum_{j=1}^p\nu_j^k \nabla^2 g_j(x^k) + \lambda^k \sum_{i = 1}^N \left[\nabla Q_\epsilon(\overline{C}^N(x^k)) \right]_i \left[ \sum_{j = 1}^m \bar{\mu}_{i_j} \nabla^2 c_j(x^k,\xi_i) \right] \nonumber \\
    &+ \lambda^k \nabla \overline{C}^N(x^k) \left[\nabla^2 Q_\epsilon(\overline{C}^N(x^k)) \right] \left[ \nabla \overline{C}^N(x^k) \right]^T,\label{eq: SQPHess}
\end{align}
where $[\nabla \overline{C}^N(x^k)]_{\boldsymbol{\cdot} i} = \nabla c(x,\xi_i)\bar{\mu}_i$ for all $i \in \{1,\ldots,N\}$,
then this algorithm exhibits quadratic local convergence, under standard regularity assumptions.

%
Our goal is to mimic this behavior in \cref{alg: JCCPSolver}.  
Using arguments similar to those in the proof of \cref{lem: KKTPoint} one can show that the steps generated by the QP \cref{eq: QPApprox} converge to the SQP steps from \cref{eq: SQP} up to first order as $x^k\to x^*$, under suitable regularity assumptions.
This suggests to choose $H^k$ defined by \cref{eq: SQPHess} for the step computation \cref{eq: QPApprox} in \cref{alg: JCCPSolver}, with the hope that the steps used in \cref{alg: JCCPSolver} approach the SQP steps for \cref{eq: SmallJCCP} and result in fast local convergence.

However, the definition of \cref{eq: SmallJCCP} requires the knowledge of the optimal solution in \cref{eq: muBar}, which is not available.  Instead, we approximate the optimal solution with the current iterate and define
$$[\bar{\mu}^k_i]_j = \frac{[\mu^k_i]_j}{ \lambda^k \left[ \nabla Q_\epsilon(C^N(x^k)) \right]_i}, \quad\forall \enskip j \in \{1,\ldots,m\}, \enskip i \in \{1,\ldots,N\}.$$
Then, the Hessian matrix in the QP \cref{eq: QPApprox} is chosen as follows:
\begin{align}
    H^k = \nabla^2 f(x^k) &+ \sum_{j=1}^p\nu_j^k \nabla^2 g_j(x^k) + \lambda^k \sum_{i = 1}^N \left[\nabla Q_\epsilon({C}^N(x^k)) \right]_i \left[ \sum_{j = 1}^m \bar{\mu}^k_{i_j} \nabla^2 c_j(x^k,\xi_i) \right] \nonumber \\
    &+ \lambda^k \nabla \overline{C}^N(x^k) \left[\nabla^2 Q_\epsilon({C}^N(x^k)) \right] \left[ \nabla \overline{C}^N(x^k) \right]^T, \label{eq: HessJCCP}
\end{align}
where $[\nabla \overline{C}^N(x^k)]_{\boldsymbol{\cdot}i} = \nabla c(x,\xi_i)\bar{\mu}^k_i$ for all $i \in \{1,\ldots,N\}$.
Note the subtle difference with \eqref{eq: SQPHess}: The argument at which derivatives of $Q_\epsilon$ are evaluated is based on the true maximum values given by the non-smooth function $C^N(x^k)$ instead of the smoothed version $\overline{C}^N(x^k)$.
With this choice of $H^k$, we observed superlinear local convergence of \cref{alg: JCCPSolver} in our experiments.


\section{Numerical results} \label{sec: NumRes}

In this section we report results demonstrating the performance of our proposed approach on two single chance-constrained (SCC) examples and one joint chance-constrained (JCC) example. The goal is to observe the performance of our method from different perspectives such as: solution quality, speed, and features related to the smoothing parameter $\epsilon$.

In all the experiments we use the $\Gamma_{\epsilon}$ function defined from $\gamma_\epsilon$ in \cref{eq: gamma_eps}, which defines a twice continuously differentiable function. This allows us to incorporate first- and second-order derivatives into the algorithms. All computations were executed on \texttt{Ubuntu 16.04} with 256GB RAM and a CPU with two \texttt{Intel Xeon} processors with 10 cores each, running at 3.10GHz. The algorithm and all experiments were implemented in \texttt{Matlab R2015b}.

In these experiments, because our choices of $\alpha$ and $N$ are such that $(1-\alpha)N \in \mathbb{Z}$, we opt to add a constant $b = 1/2$ to the left-hand side of \cref{eq: SmoothVaR} in our implementation (refer to \cref{rem: NotIntegerB}).

\subsection{Choosing the smoothing parameter} \label{sec: ChoosingEpsilon}

We introduced an approximation of problem \cref{eq: CCP} that does not need explicit information about the distribution of the random variables, their mean or their variance, and that only requires a sample to compute the constraint approximation. However, the quality of the solutions obtained from the approximation depends on the choice of the parameter $\epsilon$. As established in \cref{sec: Feasibility}, $\epsilon$ is related to the feasibility of the solutions. Moreover, in our experiments we observed that large values of $\epsilon$ are associated with fewer iterations and less variability in the solutions but at the same time lead to conservativeness, while small values work in the opposite manner. Therefore, a good choice of $\epsilon$ is of vital importance to the solution quality.

For a fixed sample, we propose a binary search algorithm to find a value of $\epsilon$ which yields a solution that is
feasible to the chance constraint, but not too conservative. We begin by setting $\ell=0$, selecting an initial $\epsilon_0$, and setting $\epsilon_{LB} = 0$ and $\epsilon_{UB} = \infty$. Then, we solve problem \cref{eq: SingleJCCP} for $\epsilon_\ell$ and estimate the true probability $p_\ell = \proba(C(x_\ell^*,\xi) \leq 0)$, where $x_\ell^*$ is the optimal solution of problem \cref{eq: SingleJCCP} using the parameter $\epsilon_\ell$. If $p_\ell > 1-\alpha$, then $\epsilon_{\ell+1} = (\epsilon_\ell + \epsilon_{LB})/2$ and $\epsilon_{UB} = \epsilon_\ell$; if $p_\ell < 1-\alpha$, then $\epsilon_{\ell+1} = (\epsilon_{UB} + \epsilon_\ell)/2$, or $\epsilon_{\ell+1} = 2\epsilon_0$ if $\epsilon_{UB} = \infty$, and $\epsilon_{LB} = \epsilon_\ell$. We now solve the problem again for $\epsilon_{\ell+1}$ and compute $p_{\ell+1}$. We continue in this fashion until the estimated probability, $p_\ell$, is within a threshold of the target probability $1-\alpha$.


One way to pick the initial $\epsilon_0$ for the binary search is as follows. First, compute $x_0^*$ as the optimal solution of the ``robust'' formulation of \cref{eq: CCP},
$$\min_{x \in X} f(x) \quad \text{s.t. } C(x,\xi_i) \leq 0 \enskip \forall i \in \{1,\ldots,N\}.$$
Then, evaluate the standard deviation of the sample $\{C(x_0^*,\xi_1),\ldots,C(x_0^*,\xi_N)\}$, and set $\epsilon_0$ to be twice the standard deviation. This way of selecting $\epsilon_0$ usually renders conservative solutions; however, because we have observed that the optimization algorithm converges more quickly for large values of $\epsilon$, we prefer to start with a large value of $\epsilon$ that is decreased afterwards. This choice of $\epsilon_0$ is motivated by the attempt to minimize the mean square error in \cref{eq: MSESmoothKernel}, for a fixed $x$. In \cite{AlemBolaGuil12}, the authors show that the best choice of $\epsilon$ is proportional to its standard deviation if everything is assumed to be normal. Thus, we define $\epsilon_0$ as a multiple of the standard deviation evaluated at $x_0^*$.

To estimate the true probability of constraint violation, we obtain an out-of-sample approximation using a sample with a large number of scenarios $N'$ and compute the quantity $\sum_{i = 1}^{N'} \mathbbm{1}(C(x,\xi_i) \leq 0) / N'$ as an approximation of $\proba(C(x,\xi) \leq 0)$.

To speed up the binary search, we use the solution $x^*_{\epsilon_k}$, with its multipliers $\lambda^*_{\epsilon_k}$ and $\bar{\mu}^*_{\epsilon_k}$, to initialize the algorithm for $\epsilon_{k + 1}$, i.e., $x_0 = x^*_{\epsilon_k}$, $\lambda_0 = \lambda^*_{\epsilon_k}$, and $\bar{\mu}_0 = \bar{\mu}^*_{\epsilon_k}$ to solve problem \cref{eq: JCCP} for $\epsilon_{k+1}$.

For the experiments in the following section, we terminate the binary search if $|\proba(C(x_\ell^*,\xi) - (1-\alpha)| \leq 10^{-4}$ or the number of bisections exceeds 10.

\subsection{Single Chance Constraints} \label{sec: NumResSCCP}

In this section, we present the outcome of two experiments using a nonlinear solver to solve our reformulation for a SCCP. Both examples are solved by \texttt{Knitro 10.1.2}. To compute $Q_\epsilon$, we use the \texttt{Matlab} function \texttt{fzero} to find the root of the nonlinear function \cref{eq: SmoothVaR} from the initial point $C_{[M]}(x)$. We also provide \texttt{Knitro} both the gradient and Hessian of $Q_\epsilon$ as given in \cref{eq: ImplicitGradient} and \cref{eq: SCCHessian}.

\subsubsection{A non-convex example}\label{sec: NonConvEg}

We first investigate the performance of the NLP reformulation for the non-convex problem
$$\min_{x, y} \enskip y \quad \text{s.t. } \proba(c(x,\xi) \leq y) \geq 0.95,$$
where $c(x,\xi) = 0.25x^4 - 1/3x^3 - x^2 + 0.2x - 19.5 + \xi_1x + \xi_2,$
and $\xi_1 \sim N(0,3)$ and $\xi_2 \sim N(0,144)$ are independent random variables. This problem is similar to the one presented in \cite{CurtWaecZava18}.

In this experiment we use a sample of size $N = 1,000$ and two different smoothing parameters: $\epsilon = 1$ and $\epsilon = 0.1$. We used 10 different feasible initial points for each $\epsilon$, given by $x_0 = (t,2.5)$, for $t$ equally spaced in the interval $[-1.5,2.5]$. The solving time for all starting points and both $\epsilon$ values is less than 0.0867 seconds. In \cref{fig: SCCP_NonConvex} we observe the results of all the different runs. In this figure, the solid line represents the boundary of the empirical feasible region, while the dotted line represents the boundary of the smooth approximation to the feasible region. All the points above the solid and dotted line are feasible for the empirical and smooth approximations, respectively. We see that the smaller value of $\epsilon$ translates into a smooth curve that follows the empirical quantile more closely, while the larger value renders a ``less noisy'' function.

\begin{figure}[htbp]
\centering
\begin{subfigure}{0.5\textwidth}
  \centering
  \includegraphics[width=0.99\linewidth]{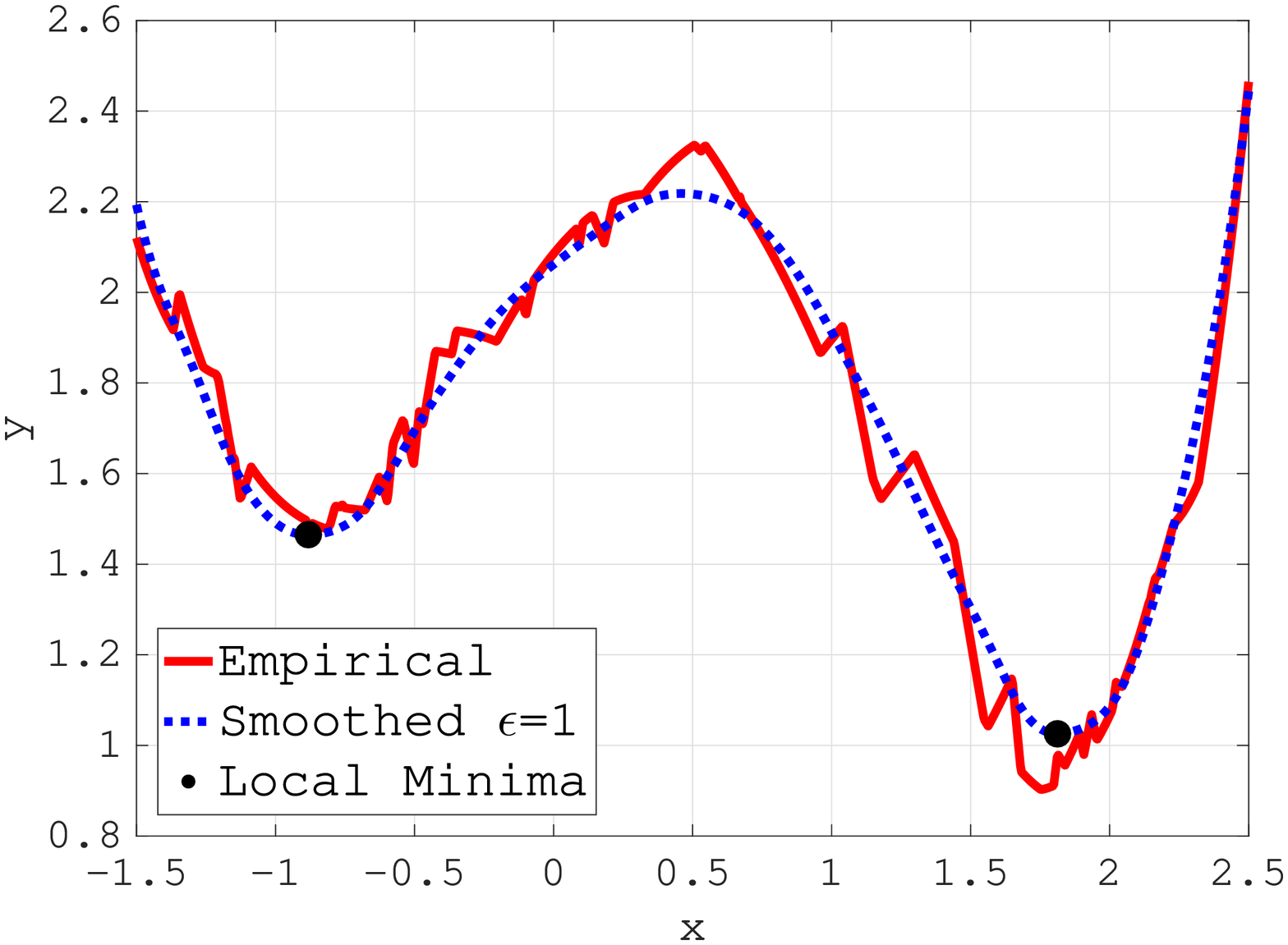}
  \caption{$\epsilon = 1$}
\end{subfigure}%
\begin{subfigure}{0.5\textwidth}
  \centering
  \includegraphics[width=0.99\linewidth]{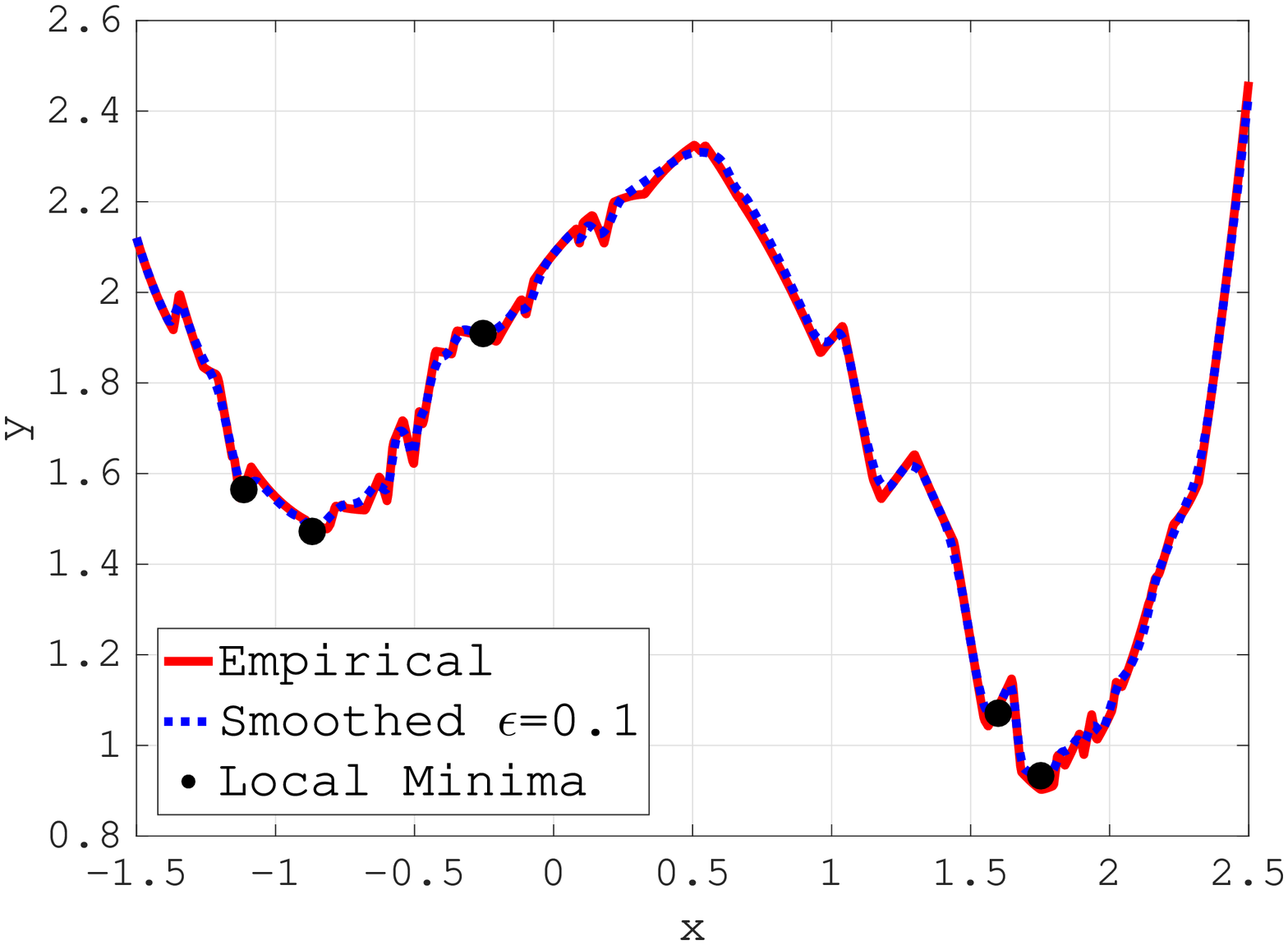}
  \caption{$\epsilon = 0.1$}
\end{subfigure}
\caption{Local minima found starting from 10 different points for two values of $\epsilon$.}
\label{fig: SCCP_NonConvex}
\end{figure}

\Cref{fig: SCCP_NonConvex} shows that our approach is designed to find local minima instead of global minima, and that the quality of the solution not only depends on the starting point but on the choice of $\epsilon$. This can be seen from the points in \cref{fig: SCCP_NonConvex}. Notice that the true problem seems to have only two local minimizers, one around $x=1.8$ and another around $x=-0.8$. However, when $\epsilon = 0.1$, the solver gets attracted to spurious local minima for three different starting points.

It is also interesting to notice that the algorithm takes longer to find a solution the smaller $\epsilon$ is. For $\epsilon = 1$, the number of iterations ranged between 6 and 10 with an average of 8.2; however, for $\epsilon = 0.1$, the number of iterations ranges between 7 and 14 with an average of 11.6. This exemplifies some of the claims made in \cref{sec: SmoothVaR} and emphasizes the importance of choosing $\epsilon$ appropriately.

\subsubsection{Reinsurance portfolio example} \label{sec: ReinsuranceSCCP}

We consider a Value-at-Risk (VaR) minimization problem for reinsurance portfolios. The decision variable $x_j$, for $j = 1,\ldots,n$, represents the fraction of the risk contract $j$ that the reinsurer is willing to take. We would like to choose the optimal fraction of the contracts such that the portfolio premium satisfies a pre-established minimum value and that the $0.95$-VaR is minimized. Each contract is subject to a stochastic loss. This data is simulated to represent risk and returns of businesses. The characteristics of these loss matrices include non-negativity, sparsity, positive skewness, and high kurtosis (this problem was constructed as in~\cite{FengWaecStau15}).

Notice that the $0.95$-VaR is defined the same way as the $0.95$-quantile, so we can therefore use $Q^{0.95}$ as the objective function. Thus, the problem can be modeled as
$$\min_{x\in [0,1]^n} Q^{0.95}(L^Tx) \quad \text{s.t. } c^Tx \geq c^*,$$
where the random variable $L_j$ represents the losses of contract $j$, $c_j$ are the market premiums of the risk contracts, and $c^* = 0.1\sum_{j = 1}^n c_j$ is the minimum level of portfolio premium. The chance-constrained reformulation of the problem is
$$ \min_{x\in [0,1]^n, z \in \mathbb{R}} z \quad \text{s.t. }  \proba(L^Tx \leq z) \geq 1-\alpha, \enskip c^Tx \geq c^*.$$
For this example, we considered $25$ contracts ($n=25$) for different sample sizes $N$, and we ran 10 replications for each $N$.

In \cref{tab: SCCP_Normal}, we present the statistics of the \texttt{Knitro} runs using formulation \cref{eq: VaRFormulation}, the quantile-constrained NLP approximation of the problem. The times reported include all NLP solves executed during the bisection search to tune the smoothing parameter $\epsilon$ described in \cref{sec: ChoosingEpsilon}. We see that the solution time grows modestly with respect to the sample size, and that the number of iterations are roughly the same even when the sample size increases.
The last column displays the average value of the tuned smoothing parameter.  Consistent with our theoretical analysis in \cref{sec: Feasibility} and the results in \cite{Azza81}, we observe that the value of $\epsilon$ that yields a solution that is ``just feasible'' decreases as the sample size increases.

\begin{table}[htbp]
\centering
\begin{tabular}{ c | c | c | c | c }
   N & Avg. time (sec) & Max. time (sec) & Avg. iter (per $\epsilon$) & Avg. $\epsilon$ \\\hline\hline
 200 &        1.1301 & 1.8285 & 13.236 & 1628.3 \\
 500 &        1.4797 & 2.6307 & 14.250 & 1135.6 \\
2000 &        10.293 & 13.643 & 13.479 & 700.68 \\
5000 &        49.943 & 113.91 & 14.073 & 554.47
\end{tabular}
\caption{Number of iterations and CPU time needed by \texttt{Knitro} to return a solution.}
\label{tab: SCCP_Normal}
\end{table}

We compare our approach to the empirical probability formulation $\widetilde{F}^N$. This formulation is solved via a mixed-integer linear programming (MIP) formulation \cite{Rusz02}, using \texttt{Gurobi} as the MIP solver with a time limit of 10 minutes. 
Mixing inequalities \cite{Lued14} were not used, because the required preprocessing took already more than 10 min for $N=2,000$.

For the solutions obtained by either method, we computed an out-of-sample approximation of the true quantile, $Q^{0.95}(L^Tx)$, via the empirical $0.95$-quantile using a sample of $N'=$100,000 observations. The results of the experiments are shown in \cref{fig: SCCP_Reinsurance}. We see that, for $N = 200$ and $N = 5,000$, the variance of the solutions obtained by the NLP solver is considerably smaller than the variance of the solutions obtained using the MIP formulation. The spread observed on the quality of the solutions from the MIP for $N = 5,000$ may be due to the 10 minute limit, since the incumbent returned by the solver is not proved to be optimal for the sample approximation for this sample size. We also observe that the solutions obtained by the NLP formulation are consistently better in all cases and are improving as the sample size grows.

\begin{figure}[htbp]
\centering
\begin{subfigure}{0.33\textwidth}
  \centering
  \includegraphics[width=1\linewidth]{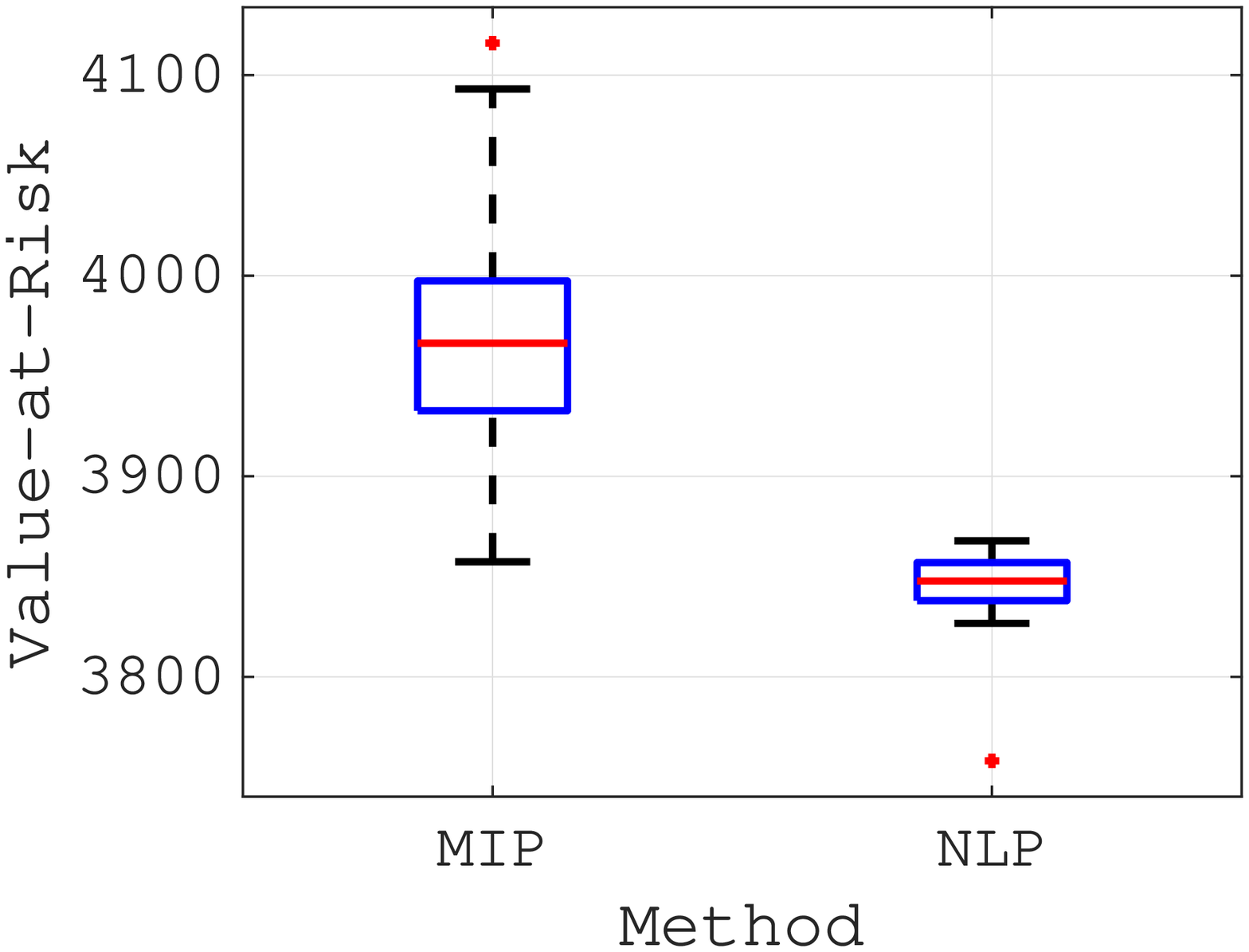}
  \caption{$N = 200$}
\end{subfigure}%
\begin{subfigure}{0.33\textwidth}
  \centering
  \includegraphics[width=1\linewidth]{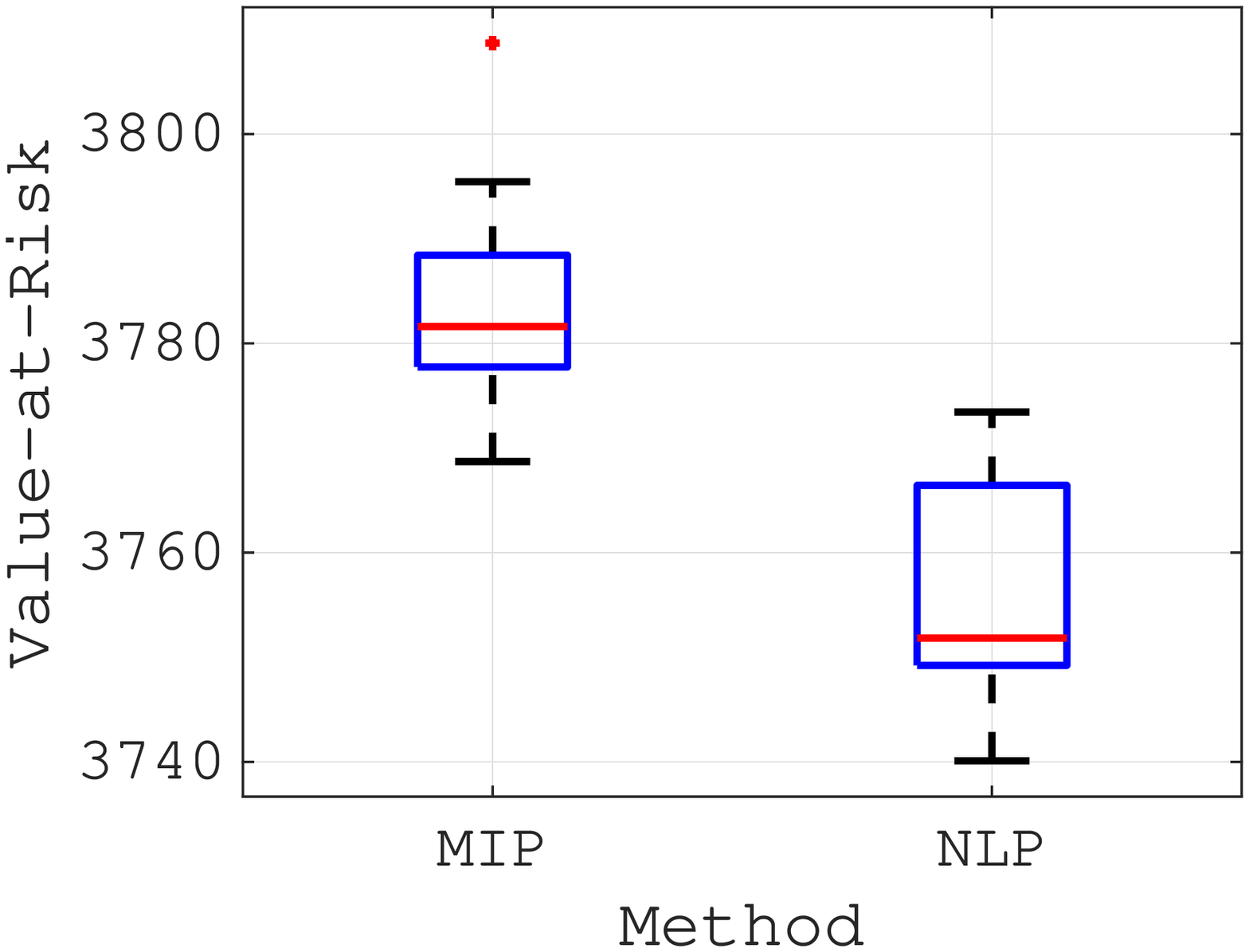}
  \caption{$N = 2,000$}
\end{subfigure} 
\begin{subfigure}{0.33\textwidth}
  \centering
  \includegraphics[width=1\linewidth]{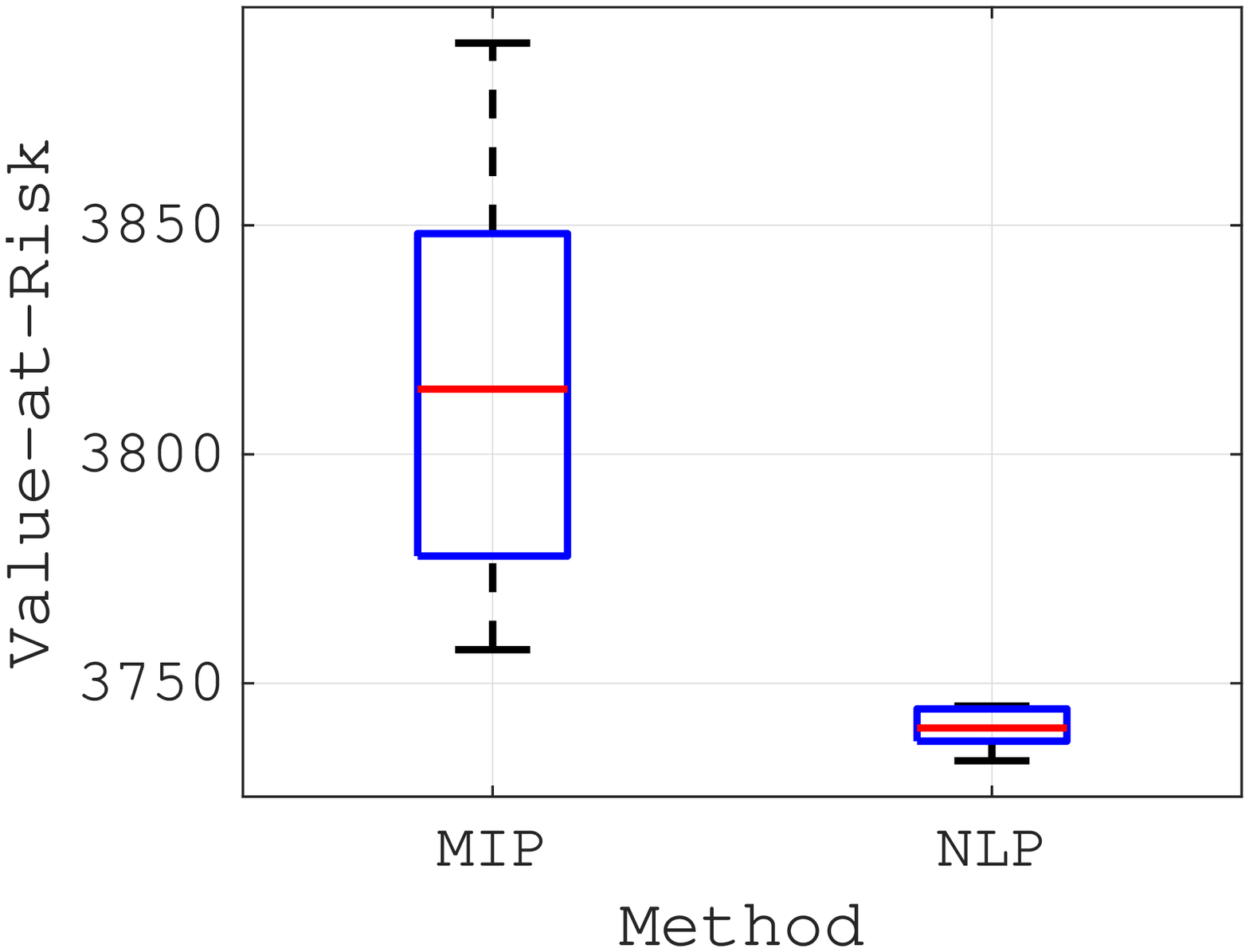}
  \caption{$N = 5,000$}
\end{subfigure}
\caption{Comparison between an MILP approach and our NLP method.}
\label{fig: SCCP_Reinsurance}
\end{figure}

We now compare the performance of the NLP solver when constraint \cref{eq: VaRConstraint} is substituted by $F_\epsilon(0;x) = \frac{1}{N}\sum_{i = 1}^N \Gamma_\epsilon(c(x,\xi_i)) \geq 1 - \alpha$. We use this example to evaluate the advantage of the quantile formulation of the constraint compared to the probability formulation. We did these experiments under the same settings as above. The solutions that the solver found with the probability formulation are the same as for the quantile formulation (within numerical accuracy), except for two cases for $N = 2,000$ and one for $N = 5,000$ in which \texttt{Knitro}'s iterates diverge. The number of iterations and the time taken by the solver is considerably larger compared to the quantile formulation as can be seen in \cref{fig: SCCP_ReinsuranceNLP}. Furthermore, while the quantile formulation requires roughly the same number of iterations for different number of scenarios, the iteration count for the probabilistic constraint grows with the sample size. Therefore, although finding the root of \cref{eq: SmoothVaR} is more expensive than using $F_\epsilon(0;x)$ directly, working with the quantile is preferable in terms of the performance of the solver.

\begin{figure}[htbp]
\centering
\begin{subfigure}{0.5\textwidth}
  \centering
  \includegraphics[width=0.99\linewidth]{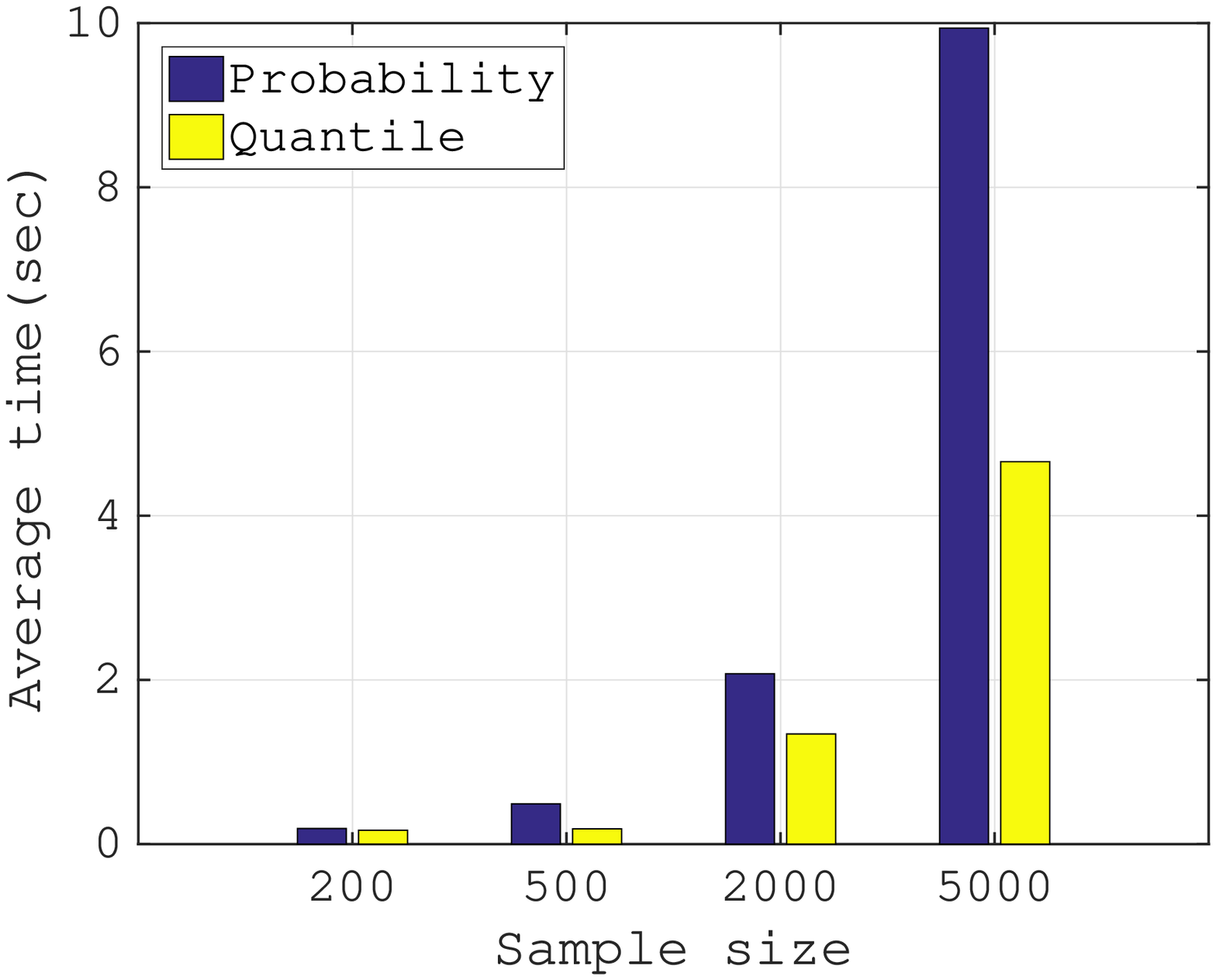}
  \caption{Average time.}
\end{subfigure}%
\begin{subfigure}{0.5\textwidth}
  \centering
  \includegraphics[width=0.99\linewidth]{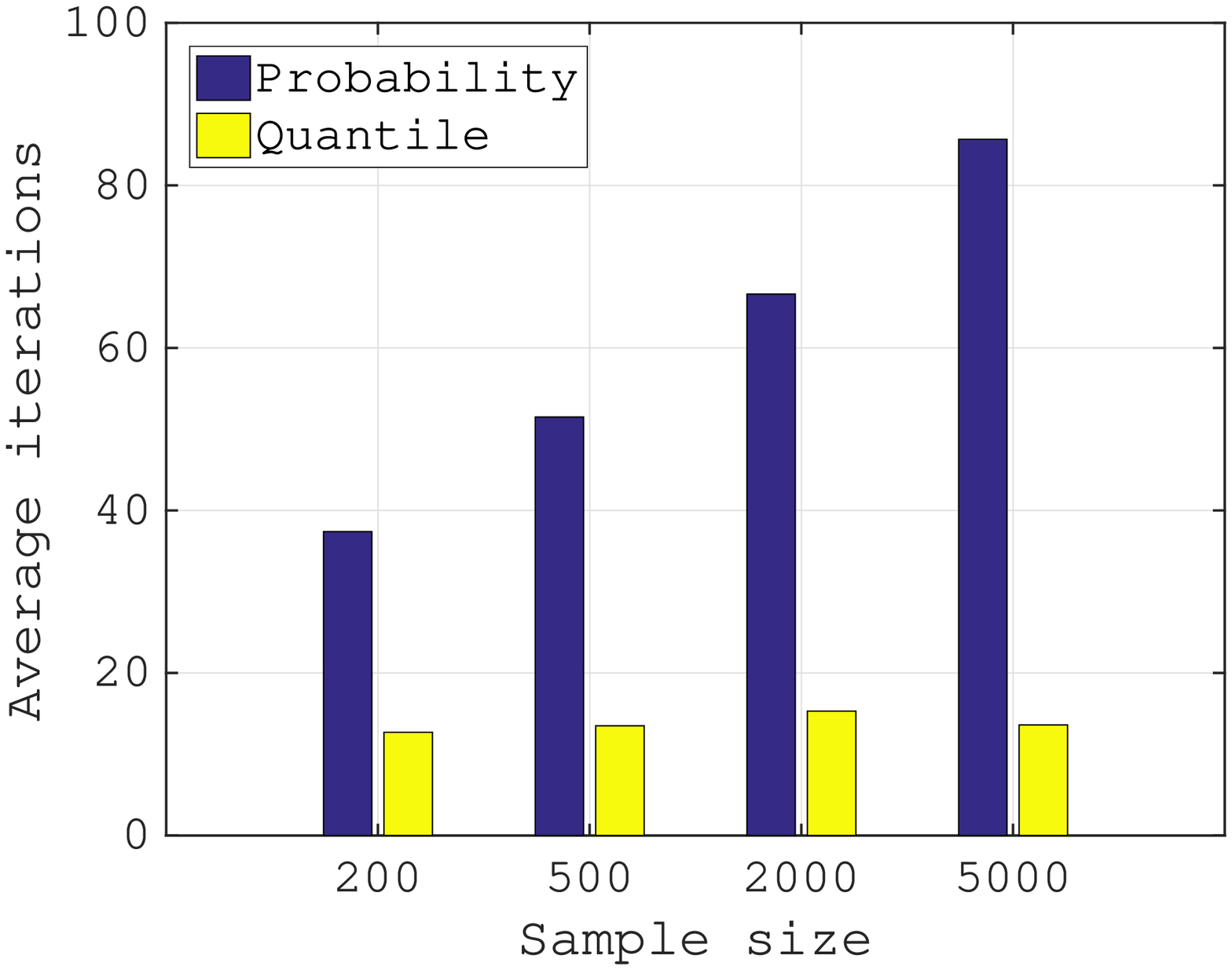}
  \caption{Average number of iterations.}
\end{subfigure}
\caption{Performance comparison between the probability and quantile approximations.}
\label{fig: SCCP_ReinsuranceNLP}
\end{figure}

\subsection{Joint Chance Constraints} \label{sec: NumResJCCP}

The last example we consider is a continuous probabilistic multi-dimensional knapsack problem
\begin{align}
\max_{x \in [0,1]} &\quad \sum_{j = 1}^n p_j x_j \label{eq: MultKnapEg}\\
\text{s.t.} &\quad \proba\left( \sum_{j = 1}^n W_{ij}x_j \leq w_i, i \in \{1,\ldots,m\} \right) \geq 0.95. \nonumber
\end{align}
The problem consists of a set of $n$ items, each with an associated profit $p_j$, and different ``weights'' for each item (for example, both volume and weight, where the volume and weight of each item are not necessarily related). Here, $W_{ij}$ represents the random $i$th ``weight'' of item $j$, and $w_i$ is the maximum ``weight'' limit for constraint $i$ ($w_i$ is a deterministic quantity). The decision variable $x_j \in [0,1]$ represents the fraction of an item that is to be included. The goal is to select the optimal fraction each item in order to maximize the profit, while satisfying all the ``weight'' constraints simultaneously at least $95\%$ of the time. The number of items considered for this experiment was $n = 20$ with $m = 10$ constraints. The randomness comes from the weights $W_{ij}$, whose distribution is described next.

We implemented \cref{alg: JCCPSolver} in \texttt{Matlab}, using \texttt{CPLEX 12.6.3} to solve the QP \cref{eq: QPApprox}. We selected the following parameters for the algorithm: $\pi = 10$, $\hat{\Delta} = 10^6$, $\Delta_0 = 1$, $\eta = 10^{-8}$, $\tau_1 = 1/2$ and $\tau_2 = 2$. We defined $H^k$ as \cref{eq: HessJCCP}, and terminate \cref{alg: JCCPSolver} when $\| \nabla \overline{\mathcal{L}}(x_k,\nu,\lambda) \|_{\infty} \leq 10^{-6}$, $g(x_k) \leq 10^{-6}\pmb{1}_m$, and $Q_\epsilon(C^N(x_k)) < 10^{-6}$.

We generate data for \cref{eq: MultKnapEg} according to the knapsack instance \textit{mk-20-10} from \cite{SongLuedKucu14}, which is based on an available deterministic knapsack instance in \cite{Beas90}. Each item $j$ is available or not according to a Bernoulli distribution. An item that is not available has weight zero in all rows. As a result, the item weights in different rows are not independent. If an item is available, then its weight in each row is normally distributed with mean equal to its weight in the deterministic instance and standard deviation equal 0.1 times the mean. For each sample size $N$, we ran 10 replications.

The results of our proposed algorithm are presented in \cref{tab: JCCP_Knapsack} under the column heading `NLP'.
The computation times include the total time for the binary search algorithm to find a tuned value of $\epsilon$. For
the solutions obtained, we computed an out-of-sample approximation of $\proba(C(x,\xi))$ using the empirical cdf with
$N'=10^6$ scenarios. First, we can see that the time it takes the algorithm to obtain a solution grows modestly with the
sample size $N$. Second, the best solution obtained from solving the problem with 10 different samples of size 100 has an
objective value nearly identical to the best solution obtained using samples of size 1,000. This indicates that solving
the approximation with multiple small samples may be more effective at obtaining good solutions than solving the approximation with a single large sample. Finally, the last row shows that the average value of the tuned smoothing parameter decreases with increased sample size, similar to our observation in \cref{tab: SCCP_Normal}

\begin{table}[htbp]
\centering
\begin{tabular}{ c | | c | c | c | c | c | c }
Sample size & \multicolumn{2}{|c|}{100} & \multicolumn{2}{|c|}{500} & \multicolumn{2}{|c }{1,000} \\ \hline
Method      & NLP      & MIP    & NLP    & MIP    & NLP     & MIP \\\hline \hline
Min. prob   & 0.9499   & 0.9445 & 0.9499 & 0.9503 & 0.9499  & 0.9502 \\
Avg. prob   & 0.9504   & 0.9591 & 0.9503 & 0.9520 & 0.9511  & 0.9514 \\
Max. prob   & 0.9521   & 0.9723 & 0.9526 & 0.9549 & 0.9591  & 0.9535 \\ \hline
Min. obj    & 5816.6   & 5771.3 & 5834.4 & 5830.3 & 5827.5  & 5831.3 \\
Avg. obj    & 5828.9   & 5806.0 & 5839.0 & 5833.2 & 5840.2  & 5836.6 \\
Max. obj    & 5842.6   & 5836.1 & 5842.6 & 5839.2 & 5842.7  & 5841.0 \\ \hline
Avg. time (s) & 8.7 & 90    & 29.3  & 168    & 37.4  & 372    \\
Max. time (s) & 24.6 & 120   & 49.8  & 180    & 80.5  & 420    \\ \hline
Avg. $\epsilon$ & 17   & -     & 9.8     &  -     & 7.5     & -            
\end{tabular}

\bigskip

\begin{tabular}{ c | | c | c | c | c }
Sample size & \multicolumn{2}{|c|}{5,000} & \multicolumn{2}{|c }{10,000} \\ \hline
Method      & NLP      & MIP    & NLP    & MIP    \\\hline \hline
Min. prob   & 0.9499   & 0.9514 & 0.9499 & 0.9506 \\
Avg. prob   & 0.9500   & 0.9562 & 0.9500 & 0.9606 \\
Max. prob   & 0.9501   & 0.9617 & 0.9503 & 0.9733 \\ \hline
Min. obj    & 5840.8   & 5818.1 & 5841.2 & 5798.0 \\
Avg. obj    & 5842.3   & 5828.7 & 5842.4 & 5819.4 \\
Max. obj    & 5842.9   & 5839.9 & 5842.9 & 5833.8 \\ \hline
Avg. time (s) & 88.265 & 330    & 324.39 & 600 \\
Max. time (s) & 163.96 & 600    & 582.53 & 600 \\ \hline
Avg. $\epsilon$ & 5.8  & -      & 4.8    &  -            
\end{tabular}

\caption{Statistics for the multi-knapsack problem.}
\label{tab: JCCP_Knapsack}
\end{table}

We also compared our solutions to the solutions obtained by solving the empirical probability formulation $\widetilde{F}^N$ via an MIP formulation. In order to obtain comparable solutions between the MIP and the NLP formulations, we tuned the risk level $\alpha$ of the MIP formulation to ensure that the solutions obtained were all feasible. This is similar to tuning the $\epsilon$ parameter in the NLP formulation. We limited the overall tuning time for the $\alpha$ parameter in the MIP to ten minutes. For each $\alpha$ trial of the MIP, we gave \texttt{Gurobi} 0.5, 0.5, 1, 5, and 10 minutes for sample size 100, 500, 1,000, 5,000, and 10,000, respectively. These time limits were chosen to give the MIP algorithm significantly more time than the NLP approach. The results are given in \cref{tab: JCCP_Knapsack} under the column heading `MIP'. It can be seen that the MIP solutions have generally lower optimal values and larger variability, the same behavior observed in \cref{sec: ReinsuranceSCCP}.  On the other hand, solutions obtained from \cref{alg: JCCPSolver} consistently present larger objective function values, less variability, and are solved well under 3 minutes for most sample sizes (except $N = 10,000$). It is possible that giving the MIP algorithm more time could lead to better solutions, but these results clearly indicate that our NLP approach is more effective at finding good solutions quickly.  Furthermore, the results show that the binary search for $\epsilon$ can attain the desired risk-level more accurately than changing the $\alpha$ parameter, given the time limit.

\section{Summary} \label{sec:concl}

In this paper we developed a technique to solve nonlinear problems with probabilistic constraints. We first presented an approximation for single chance constraints with the following characteristics: (1) it can be represented as a single smooth constraint, (2) it does not require information about the specific distribution of the random variables or information about the moments since it solely based on sampling; it only needs to satisfy continuity of the distribution, (3) it does not restrict the constraints to be convex, (4) it can be handled by standard continuous nonlinear programming techniques and solvers, and (5) it avoids the combinatorial complexity of a branch and bound search, and scales well with the number of variables and samples.

We also presented an extension of the SCC formulation that is able to handle joint chance constraints. We constructed a specialized trust-region solver for JCCPs utilizing a Hessian matrix that shows fast local convergence in our experiments.


\bibliographystyle{siamplain}
\bibliography{references}

\begin{thebibliography}{10}

\bibitem{AlemBolaGuil12}
{\sc R.~Alemany, C.~Bolanc{\'e}, and M.~Guill{\'e}n}, {\em A nonparametric
  approach to calculating value-at-risk}, Insurance: Mathematics and Economics,
  52 (2013), pp.~255--262.

\bibitem{Azza81}
{\sc A.~Azzalini}, {\em A note on the estimation of a distribution function and
  quantiles by a kernel method}, Biometrika, 68 (1981), pp.~326--328.

\bibitem{Beas90}
{\sc J.~E. Beasley}, {\em Or-library: Distributing test problems by electronic
  mail}, Journal of the Operational Research Society, 41 (1990),
  pp.~1069--1072.

\bibitem{BienCherHarn14}
{\sc D.~Bienstock, M.~Chertkov, and S.~Harnett}, {\em Chance-constrained
  optimal power flow: Risk-aware network control under uncertainty}, {SIAM}
  Review, 56 (2014), pp.~461--495.

\bibitem{BremHenrMoel15}
{\sc I.~Bremer, R.~Henrion, and A.~Moeller}, {\em Probabilistic constraints via
  sqp solver: application to a renewable energy management problem},
  Computational Management Science, 12 (2015), pp.~435--459.

\bibitem{CalaCamp05}
{\sc G.~Calafiore and M.~C. Campi}, {\em Uncertain convex programs: randomized
  solutions and confidence levels}, Mathematical Programming, 102 (2005),
  pp.~25--46.

\bibitem{CampGaraPran09}
{\sc M.~C. Campi, S.~Garatti, and M.~Prandini}, {\em The scenario approach for
  systems and control design}, Annual Reviews in Control, 33 (2009),
  pp.~149--157.

\bibitem{CaoZava17}
{\sc Y.~Cao and V.~Zavala}, {\em A sigmoidal approximation for
  chance-constrained nonlinear programs}, 2017,
  \url{http://www.optimization-online.org/DB_FILE/2017/10/6236.pdf}.

\bibitem{CharCoopSymo58}
{\sc A.~Charnes, W.~W. Cooper, and G.~H. Symonds}, {\em Cost horizons and
  certainty equivalents: an approach to stochastic programming of heating oil},
  Management Science, 4 (1958), pp.~235--263.

\bibitem{ConnGoulToin00}
{\sc A.~R. Conn, N.~I. Gould, and P.~L. Toint}, {\em Trust region methods},
  vol.~1, SIAM, 2000.

\bibitem{CurtOver12}
{\sc F.~E. Curtis and M.~L. Overton}, {\em A sequential quadratic programming
  algorithm for nonconvex, nonsmooth constrained optimization}, SIAM Journal on
  Optimization, 22 (2012), pp.~474--500.

\bibitem{CurtWaecZava18}
{\sc F.~E. Curtis, A.~W\"achter, and V.~M. Zavala}, {\em A sequential algorithm
  for solving nonlinear optimization problems with chance constraints}, SIAM
  Journal on Optimization, 28 (2018), pp.~930--958.

\bibitem{Deak00}
{\sc I.~De{\'a}k}, {\em Subroutines for computing normal probabilities of
  sets—computer experiences}, Annals of Operations Research, 100 (2000),
  pp.~103--122.

\bibitem{FengWaecStau15}
{\sc M.~Feng, A.~W{\"a}chter, and J.~Staum}, {\em Practical algorithms for
  value-at-risk portfolio optimization problems}, Quantitative Finance Letters,
  3 (2015), pp.~1--9.

\bibitem{GeleHoffKlop17}
{\sc A.~Geletu, A.~Hoffmann, M.~Kloppel, and P.~Li}, {\em An inner-outer
  approximation approach to chance constrained optimization}, SIAM Journal on
  Optimization, 27 (2017), pp.~1834--1857.

\bibitem{GenzBret09}
{\sc A.~Genz and F.~Bretz}, {\em Computation of multivariate normal and t
  probabilities}, Springer Science \& Business Media, 2009.

\bibitem{hantoute2019subdifferential}
{\sc A.~Hantoute, R.~Henrion, and P.~P{\'e}rez-Aros}, {\em Subdifferential
  characterization of probability functions under {G}aussian distribution},
  Mathematical Programming, 174 (2019), pp.~167--194.

\bibitem{Heit19}
{\sc H.~Heitsch}, {\em On probabilistic capacity maximization in a stationary
  gas network}, Optimization,  (2019), pp.~1--30.

\bibitem{Henr07}
{\sc R.~Henrion}, {\em Structural properties of linear probabilistic
  constraints}, Optimization, 56 (2007), pp.~425--440.

\bibitem{Henr11}
{\sc R.~Henrion}, {\em A critical note on empirical (sample average, {M}onte
  {C}arlo) approximation of solutions to chance constrained programs}, in IFIP
  Conference on System Modeling and Optimization, Springer, 2011, pp.~25--37.

\bibitem{HenrLiMoll01}
{\sc R.~Henrion, P.~Li, A.~M{\"o}ller, M.~C. Steinbach, M.~Wendt, and
  G.~Wozny}, {\em Stochastic optimization for operating chemical processes
  under uncertainty}, in Online Optimization of Large Scale Systems, Springer,
  2001, pp.~457--478.

\bibitem{HenrMoll03}
{\sc R.~Henrion and A.~M{\"o}ller}, {\em Optimization of a continuous
  distillation process under random inflow rate}, Computers \& Mathematics with
  Applications, 45 (2003), pp.~247--262.

\bibitem{HiriLema93}
{\sc J.-B. Hiriart-Urruty and C.~Lemar{\'e}chal}, {\em Convex analysis and
  minimization algorithms II}, Springer-Verlag, 1993.

\bibitem{Hoef63}
{\sc W.~Hoeffding}, {\em Probability inequalities for sums of bounded random
  variables}, Journal of the American Statistical Association, 58 (1963),
  pp.~13--30.

\bibitem{HongYangZhan11}
{\sc L.~J. Hong, Y.~Yang, and L.~Zhang}, {\em Sequential convex approximations
  to joint chance constrained programs: A {M}onte {C}arlo approach}, Operations
  Research, 59 (2011), pp.~617--630.

\bibitem{HuHongZhan13}
{\sc Z.~Hu, L.~J. Hong, and L.~Zhang}, {\em A smooth {M}onte {C}arlo approach
  to joint chance-constrained programs}, IIE Transactions, 45 (2013),
  pp.~716--735.

\bibitem{LodiMalaNann16}
{\sc A.~Lodi, E.~Malaguti, G.~Nannicini, and D.~Thomopulos}, {\em Nonlinear
  chance-constrained problems with applications to hydro scheduling}, tech.
  report, IBM Research Report RC25594 (WAT1602-046), 2016.

\bibitem{Lued14}
{\sc J.~Luedtke}, {\em A branch-and-cut decomposition algorithm for solving
  chance-constrained mathematical programs with finite support}, Mathematical
  Programming, 146 (2014), pp.~219--244.

\bibitem{LuedAhme08}
{\sc J.~Luedtke and S.~Ahmed}, {\em A sample approximation approach for
  optimization with probabilistic constraints}, SIAM Journal on Optimization,
  19 (2008), pp.~674--699.

\bibitem{LuedAhmeNemh10}
{\sc J.~Luedtke, S.~Ahmed, and G.~L. Nemhauser}, {\em An integer programming
  approach for linear programs with probabilistic constraints}, Mathematical
  Programming, 122 (2010), pp.~247--272.

\bibitem{Mayer00}
{\sc J.~Mayer}, {\em On the numerical solution of jointly chance constrained
  problems}, in Probabilistic constrained optimization, Springer, 2000,
  pp.~220--235.

\bibitem{MolzRoal18}
{\sc D.~K. Molzahn and L.~A. Roald}, {\em Towards an ac optimal power flow
  algorithm with robust feasibility guarantees}, in 2018 Power Systems
  Computation Conference (PSCC), IEEE, 2018, pp.~1--7.

\bibitem{NemiShap06}
{\sc A.~Nemirovski and A.~Shapiro}, {\em Convex approximations of chance
  constrained programs}, SIAM Journal on Optimization, 17 (2006), pp.~969--996.

\bibitem{NemiShap06_SA}
{\sc A.~Nemirovski and A.~Shapiro}, {\em Scenario approximations of chance
  constraints}, in Probabilistic and Randomized Methods for Design Under
  Uncertainty, Springer, 2006, pp.~3--47.

\bibitem{NoceWrig06}
{\sc J.~Nocedal and S.~J. Wright}, {\em Numerical optimization}, Springer,
  2006.

\bibitem{PagnAhmeShap09}
{\sc B.~Pagnoncelli, S.~Ahmed, and A.~Shapiro}, {\em Sample average
  approximation method for chance constrained programming: theory and
  applications}, Journal of Optimization Theory and Applications, 142 (2009),
  pp.~399--416.

\bibitem{Prek03}
{\sc A.~Pr{\'e}kopa}, {\em Probabilistic programming}, Handbooks in Operations
  Research and Management Science, 10 (2003), pp.~267--351.

\bibitem{Rusz02}
{\sc A.~Ruszczy{\'n}ski}, {\em Probabilistic programming with discrete
  distributions and precedence constrained knapsack polyhedra}, Mathematical
  Programming, 93 (2002), pp.~195--215.

\bibitem{ScotTapiThom77}
{\sc D.~W. Scott, R.~A. Tapia, and J.~R. Thompson}, {\em Kernel density
  estimation revisited}, Nonlinear Analysis: Theory, Methods \& Applications, 1
  (1977), pp.~339--372.

\bibitem{ShapDentRusz09}
{\sc A.~Shapiro, D.~Dentcheva, and A.~Ruszczy{\'n}ski}, {\em Lectures on
  stochastic programming: Modeling and theory}, SIAM, second~ed., 2014.

\bibitem{SongLuedKucu14}
{\sc Y.~Song, J.~R. Luedtke, and S.~K{\"u}{\c{c}}{\"u}kyavuz}, {\em
  Chance-constrained binary packing problems}, INFORMS Journal on Computing, 26
  (2014), pp.~735--747.

\bibitem{Szan88}
{\sc T.~Sz{\'a}ntai}, {\em A computer code for solution of
  probabilistic-constrained stochastic programming problems}, Numerical
  techniques for stochastic optimization,  (1988), pp.~229--235.

\bibitem{VanackAlekMuno18}
{\sc W.~Van~Ackooij, I.~Aleksovska, and M.~Munoz-Zuniga}, {\em (sub-)
  differentiability of probability functions with elliptical distributions},
  Set-Valued and Variational Analysis, 26 (2018), pp.~887--910.

\bibitem{VanackHenr14}
{\sc W.~Van~Ackooij and R.~Henrion}, {\em Gradient formulae for nonlinear
  probabilistic constraints with {G}aussian and {G}aussian-like distributions},
  SIAM Journal on Optimization, 24 (2014), pp.~1864--1889.

\bibitem{VanacketAl14}
{\sc W.~van Ackooij, R.~Henrion, A.~M{\"o}ller, and R.~Zorgati}, {\em Joint
  chance constrained programming for hydro reservoir management}, Optimization
  and Engineering, 15 (2014), pp.~509--531.

\bibitem{VanackSaga14}
{\sc W.~van Ackooij and C.~Sagastiz{\'a}bal}, {\em Constrained bundle methods
  for upper inexact oracles with application to joint chance constrained energy
  problems}, SIAM Journal on Optimization, 24 (2014), pp.~733--765.

\end{thebibliography}
\end{document}